\documentclass[11pt]{article}

\usepackage{authblk}
\usepackage{graphicx}
\usepackage{amsthm,amsfonts,amssymb,amsmath}
\usepackage{tikz}
\usepackage[margin=1in]{geometry}
\usepackage[numbers,square]{natbib}
\usepackage[ruled,lined]{algorithm2e}
\SetKw{KwSet}{Set}
\SetKw{KwSolve}{Solve}
\usepackage{bbm}
\usepackage{moresize}
\usepackage{multirow,rotating,array}
\usepackage{hyperref}
\hypersetup{backref,colorlinks=true,citecolor=blue,linkcolor=blue}

\usepackage{hypernat}

\definecolor{greenmain}{HTML}{00AF64}
\definecolor{bluemain}{HTML}{0B61A4}
\definecolor{orangemain}{HTML}{FF9200}
\definecolor{redmain}{HTML}{FF4900}

\usepackage{aliascnt}
\newtheorem{theorem}{Theorem}

\newaliascnt{lemma}{theorem}
\newtheorem{lemma}[lemma]{Lemma}
\aliascntresetthe{lemma}

\newaliascnt{cor}{theorem}
\newtheorem{corollary}[cor]{Corollary}
\aliascntresetthe{cor}

\newaliascnt{def}{theorem}
\newtheorem{definition}[def]{Definition}
\aliascntresetthe{def}

\newtheorem{assumption}{Assumption}

\newcommand{\F}{\mathcal{F}}
\newcommand{\E}{\mathbb{E}}
\renewcommand{\P}{\mathbb{P}}
\newcommand{\R}{\mathbb{R}}

\newcommand{\Y}{\mathcal{Y}}

\newcommand{\norm}[1]{\lVert #1 \rVert}

\renewcommand{\hat}[1]{\widehat{#1}}

\DeclareMathOperator*{\argmin}{argmin}
\DeclareMathOperator*{\vcd}{\textsc{vcd}}
\DeclareMathOperator*{\sgn}{sgn}
\newcommand{\indicator}{\mathbf{1}}
\renewcommand{\eqref}[1]{(\ref{eq:#1})}
\newcommand{\mnorm}[1]{\ell\left(#1\right)}

\newcommand{\dplus}{{d}}
\newcommand{\GF}{G}

\newcommand{\email}[1]{\href{mailto:#1}{#1}}

\newcommand{\Yij}[2]{Y_{#1:#2}}
\newcommand{\Yin}{\Yij{1}{n}}
\newcommand{\Sij}[2]{\sigma_{#1:#2}}

\newcommand{\Pij}[2]{\P_{#1:#2}}

\newcommand{\Yinf}{Y_\infty}
\newcommand{\Sinf}{\sigma_\infty}
\newcommand{\Pinf}{\P_\infty}

\newcommand{\Yini}{\Yij{1}{n+1}}
\newcommand{\CF}{\mathcal{C}_{\mathcal{F}}}
\newcommand{\CloF}{\mathcal{C}_{\ell \circ \mathcal{F}}}

\newcommand{\BigChar}[1]{\parbox{11pt}{\HUGE #1}}

\begin{document}

\title{Nonparametric risk bounds for time-series forecasting\footnote{Email:
    \email{dajmcdon@indiana.edu}, \email{cshalizi@cmu.edu},
    \email{mark@cmu.edu}. This work is partially supported by grants
    from the Institute for New 
  Economic Thinking.  DJM is partially supported by NSF grant 
  DMS1407439. CRS is partially supported by NIH Grant  R01
  NS047493 and by NSF grants DMS1207759 and DMS 1418124. The authors wish to thank David
  N. DeJong, Larry Wasserman, Alessandro Rinaldo and Darren Homrighausen for
  valuable suggestions on early drafts.}}
\author[$\dag$]{Daniel J. McDonald}
\author[$\S\ddagger$]{Cosma Rohilla Shalizi}
\author[$\S$]{Mark Schervish}
\affil[$\dag$]{Department of Statistics, Indiana
  University Bloomington}
\affil[$\S$]{Department of Statistics, Carnegie Mellon University}
\affil[$\ddag$] {Santa Fe Institute}
\date{Version: \today}


\maketitle

\begin{abstract}
  We derive generalization error bounds for traditional time-series forecasting
  models. Our results hold for many standard forecasting tools including
  autoregressive models, moving average models, and, more generally, linear
  state-space models. These non-asymptotic bounds need only weak assumptions on
  the data-generating process, yet allow forecasters to select among competing
  models and to guarantee, with high probability, that their chosen model will
  perform well. We motivate our techniques with and apply them to standard
  economic and financial forecasting tools---a GARCH model for predicting
  equity volatility and a dynamic stochastic general equilibrium model (DSGE),
  the standard tool in macroeconomic forecasting. We demonstrate in particular
  how our techniques can aid forecasters and policy makers in choosing models
  which behave well under uncertainty and mis-specification.

\end{abstract}

\section{Introduction}
\label{sec:introduction}

Generalization error bounds are probabilistically valid, non-asymptotic tools
for characterizing the predictive ability of forecasting models. This
methodology is fundamentally about choosing particular prediction functions out
of some class of plausible alternatives so that, with high reliability, the
resulting predictions will be nearly as accurate as possible (``probably
approximately correct'').  While many of these results are aimed at
classification problems with independent and identically distributed (i.i.d.)
data, this paper adapts and extends these methods to time-series models, so
that economic and financial forecasting techniques can be evaluated rigorously.
In particular, these methods control the expected accuracy of future
predictions from mis-specified models based on finite samples.  This allows for
immediate model comparisons which neither appeal to asymptotics nor make strong
assumptions about the data-generating process, in stark contrast to such
popular model-selection tools as AIC.

To fix ideas, imagine i.i.d.\ data $((Y_1,X_1),\ldots,(Y_n,X_n))$ with
$(Y_i,X_i)\in \mathcal{Y}\times\mathcal{X}$, some prediction function $f :
\mathcal{X}\rightarrow\mathcal{Y}$, and a loss function $\ell :
\mathcal{Y}\times\mathcal{Y} \rightarrow \R^+$ which measures the cost of bad
predictions.  The \emph{generalization error} or \emph{risk} of $f$ is
\begin{equation}
  \label{eq:risk-in-general}
  R(f) := \E[\ell(Y, f(X))]
\end{equation}
where the expectation is taken with respect to $\P$, the joint distribution of
$(Y,X)$. The generalization error measures the inaccuracy of our predictions
when we use $f$ on future data, making it a natural criterion for model
selection, and a target for performance guarantees.  To actually calculate the
risk, we would need to know the data-generating distribution $\P$ and have a
single fixed prediction function $f$, neither of which is common.  Because
explicitly calculating the risk is infeasible, forecasters typically try to
estimate it, which calls for detailed assumptions on $\P$.  The alternative we
employ here is to find upper bounds on risk which hold uniformly over large
classes of models $\F$ from which some particular $f$ is chosen, possibly in a
data-dependent way, and uniformly over distributions $\P$.

Our main results in \autoref{sec:risk-bounds} assert that for wide classes of
time-series models, the expected cost of poor predictions is bounded by the
model's in-sample performance inflated by a term which balances the amount of
observed data with the complexity of the model. The bound holds with high
probability under the unknown distribution $\P$ assuming only mild
conditions---existence of some moments, stationarity, and the decay of temporal
dependence as data points become widely separated in time.
We give applications in \autoref{sec:examples}.

Our goal in this paper is to provide general bounds for common
time-series models with unbounded loss functions, no explicit regularization,
and potential dependence on the entirety of the observed data.
The bounds we derive here are the first of their kind for the time-series
models typically used in applied settings---finance, economics, engineering,
etc.---as well as covering models more common in machine learning. In
particular, we derive results for non-linear models which depend only on a
fixed quantity of recent data and linear time invariant systems,
state-space models, which use the entire past to predict new
data. These results however do not cover, e.g.\ HMMs in the strictest
sense, as they require absolutely continuous latent states rather than
discrete valued ones.

The remainder of this paper is structured as follows.
\autoref{sec:learning-theory} provides motivation and background for our
results, giving intuition in the i.i.d.\ setting by focusing on concentration of
measure ideas and characterizations of model
complexity. \autoref{sec:dependence} gives the explicit assumptions we make and
describes how to leverage powerful ideas from time series to generalize the i.i.d.\
methods.  \autoref{sec:line-time-invar} introduces linear
time-invariant systems and discusses how such forecasters are
different from, e.g., autoregressive models. \autoref{sec:risk-bounds}
states and proves risk bounds for the time-series forecasting setting,
while we demonstrate how to use the results in 
\autoref{sec:examples} and give some properties of those results in
\autoref{sec:how-loose-are}. Finally, \autoref{sec:conclusion} concludes and
illustrates the path toward generalizing our methods to more elaborate model
classes.

\subsection{Related Work}
\label{sec:related-work}

\citet{Yu-rates-of-convergence} and
\citet{Nobel-Dembo-uniform-laws-of-averages} showed that it was possible to
transfer some i.i.d.\ results to $\beta$-mixing sequences, albeit without explicit
rates of convergence \citep[c.f.][\S
  3.4]{Vidyasagar-on-learning-and-generalization}.
\citet{Meir-nonparametric-time-series,Vidyasagar-on-learning-and-generalization,Karandikar-Vidyasagar-rates-of-UCEM}
presented bounds for model classes with finite covering numbers (a sufficient
condition for which is finite VC-dimension of a related class) but additionally
require that $\Y$ is compact and $\ell$ is
bounded. Early work in signal processing \citep{ModhaMasry1998}
proposes predictors based on sequences of parametric models of
increasing memory which minimize a complexity regularized least
squares criterion and establish that these predictors deliver the same
statistical performance as oracle
predictors. \citet{SteinwartChristmann2009} prove an oracle inequality
regularized ERM algorithms when observations are $\alpha$-mixing which
are close to the optimal i.i.d. rates. 
\citet{Mohri-Rostamizadeh-rademacher-for-non-iid} give results using
Rademacher complexity which are both tighter than those using VC-dimension or
covering numbers as well as being computable from the data in many
cases. \citet{Mohri-Rostamizdaeh-stability-bounds} and \citet{AgarwalDuchi2013} consider another family of
bounds for $\phi$-mixing and $\beta$-mixing sequences when the predictors are
algorithmically stable. Many classes of common machine learning algorithms are
amenable to either Rademacher or algorithmic-stability bounds:
Kernel-regularized methods, support-vector machines, relative-entropy based
regularization, and kernel ridge regression among others. However, methods
common to time-series such as AR models, ARIMA models, ARCH and GARCH
models \citep{Engle-on-ARCH,Engle-GARCH-101},
state-space models, and other Box-Jenkins type predictors are not because they
are not explicitly regularized, the loss functions are not bounded, and the
predictions can depend on more than simply a fixed dimensional
past. \cite{risk-bounds-for-ar-models} shows that stationarity alone can be
used to impose a kernel-type regularization on an AR model, and hence,
following the results of \citet{Mohri-Rostamizadeh-rademacher-for-non-iid}, is
amenable to Rademacher complexity for a bounded loss function. 

Other
dependence conditions apart from stationary and strong mixing are also
considered in the literature. \citet{AlquierWintenberger2012} develop oracle
inequalities and model selection procedures for linear models, neural
networks, and non-parametric 
autoregressions when observations come from causal Bernoulli shifts or
bounded, weakly-dependent processes. Under the same weak-dependence
conditions, \citet{AlquierLi2014} extends this result to convex
Lipschitz loss functions and examines forecasting of the French GDP.
Finally, recent work by \citet{KuznetsovMohri2014} examines
both average-path generalization and path-dependent generalization for
certain types of non-stationary mixing processes and derives
Rademacher complexity bounds.


\section{Statistical learning theory for i.i.d.\ data}
\label{sec:learning-theory}

Our goal is to control the risk of predictive models, that is, their expected
inaccuracy on new data from the same stochastic source as the data used to fit
the model. To orient readers, we present some standard results for i.i.d.\ data,
which are adapted to the dependent setting in \autoref{sec:risk-bounds}.

Let $f: \mathcal{X} \rightarrow \mathcal{Y}$ be some function used for making
predictions of $Y$ from $X$.  We define a loss function $\ell:
\mathcal{Y}\times\mathcal{Y} \rightarrow \R^+$ which measures the cost of
making poor predictions. Throughout this paper, we will assume that
$\ell(y,y')$ is a function solely of the difference $y-y'$ where $\ell(\cdot)$
is nonnegative and $\ell(0)=0$; we take the liberty of denoting that function
$\ell(y-y')$. Then the risk of any predictor $f \in \F$ (where $f$ is fixed
independently of the data) is given by
\begin{equation}
  \label{eq:risk-as-expected-loss}
  R(f) = \E\left[\mnorm{Y -f(X)}\right],
\end{equation}
where $(X,Y)\sim \P$.  The risk or generalization error is the expected cost of
using $f$ to predict $Y$ from $X$ on a new observation.

Since the true distribution $\P$ is unknown, so is $R(f)$, but we can try to
estimate it based on our observed data.  The \emph{training error} or
\emph{empirical risk} of $f$ is
\begin{equation}
\label{eq:empirical-risk}
\widehat{R}_n(f) := \frac{1}{n}\sum_{i=1}^{n}\mnorm{Y_{i}-f(X_i)}.
\end{equation}
In other words, the in-sample training error, $\widehat{R}_n(f)$, is the
average loss over the actual training points.  For any given $f$, we can bound
$R(f)$ in terms of $\widehat{R}_n(f)$ using deviation inequalities, as
illustrated below.

When we use the data to chose an $\hat{f}$ from $\F$, we would like to bound
$R(\hat{f})$.  To do so, we must consider not just $\widehat{R}_n(\hat{f})$,
but also the size, in some sense, of $\F$.  There are a number of measures for
the size or capacity of a model many of which lead to learning theoretic risk
bounds.  Algorithmic stability
\citep{Kearns-Ron-algorithmic-stability,Bousquet-Elisseeff-stability}
quantifies the sensitivity of the chosen function to small perturbations to the
data.  Similarly, maximal discrepancy \citep{Vapnik-nature} asks how different
the predictions could be if two functions are chosen using two separate data
sets.  A more direct, functional-analytic approach partitions $\F$ into
equivalence classes under some metric, leading to covering numbers
\citep{Pollard-convergence,Pollard-empirical-processes}.  Rademacher complexity
\citep{Bartlett-Mendelson-on-Rademacher-complexity,koltchinskii2002}
directly describes a 
model's ability to fit random noise. We focus on a measure which is both
intuitive and powerful: Vapnik-Chervonenkis (VC) dimension
\citep{Vapnik-Chervonenkis-1971,Vapnik-nature}.

VC dimension starts as an idea about collections of sets.
\begin{definition}
  Let $\mathbb{U}$ be some (infinite) set and $S$ a finite-cardinality subset of
  $\mathbb{U}$. Let $\mathcal{C}$ be a family of subsets of $\mathbb{U}$. We
  say that $\mathcal{C}$ \emph{shatters} $S$ if for every $S' \subseteq S$,
  $\exists C \in\mathcal{C}$ such that $S' = S \cap C$.
\end{definition}
Essentially, $\mathcal{C}$ can shatter a set $S$ if it can pick out every
subset of points in $S$. This says that the collection $\mathcal{C}$ is very
complicated or flexible. The cardinality of the largest set $S$ that can be
shattered by $\mathcal{C}$ is the latter's VC dimension.
\begin{definition}
  [VC dimension]
  The \emph{Vapnik-Chervonenkis (VC) dimension} of a collection
  $\mathcal{C}$ of subsets of $\mathbb{U}$ is
  \begin{equation}
    \vcd(\mathcal{C}) := \sup \{ |S| : S\subseteq \mathbb{U}\mbox{ and
      $S$ is shattered by
    } \mathcal{C} \}.
  \end{equation}
\end{definition}
To see why this is a ``dimension'', we need one more notion.
\begin{definition}
  [Growth function] The \emph{growth function} $G(n,\mathcal{C})$ of a
  collection $\mathcal{C}$ of subsets of $\mathbb{U}$ is the maximum
  number of
  subsets which can be formed by intersecting a set $S \subset \mathbb{U}$ of
  cardinality $n$ with $\mathcal{C}$,
  \begin{equation}
    G(n,\mathcal{C}) := \sup_{S\subset U~:~ |S|=n}{|S\wedge \mathcal{C}|},
  \end{equation}
  where $ \mathcal{A} \wedge \mathcal{B}$ is the class of all sets $A \cap B$, $A \in \mathcal{A}$, $B \in \mathcal{B}$.
\end{definition}
The growth function counts how many {\em effectively} distinct sets the
collection contains, when we can only observe what is going on at $n$ points,
not all of $\mathbb{U}$.  If $n \leq \vcd(\mathcal{C})$, then from the
definitions $G(n,\mathcal{C}) = 2^n$, If the VC dimension is finite, however,
and $n > \vcd(\mathcal{C})$, then $G(n,\mathcal{C}) < 2^n$, and in fact it can
be shown \citep{Vapnik-Chervonenkis-1971} that
\begin{equation}
  \label{eq:sauer}
  G(n,\mathcal{C}) \leq
  \left(\frac{en}{\vcd(\mathcal{C})}\right)^{\vcd(\mathcal{C})} \leq
  (n+1)^{\vcd(\mathcal{C})}.
\end{equation}
This polynomial growth of capacity with $n$ is why $\vcd$ is a ``dimension''.

Using VC dimension to measure the capacity of function classes is
straightforward.  Define the indicator function $\indicator_A(x)$ to take the
value 1 if $x\in A$ and 0 otherwise. Suppose that $f\in\F$,
$f:\mathbb{U}\rightarrow\R$. Each $f$ corresponds to the set
\begin{equation}
  C_f = \{ (u,u_0) \in \mathbb{U}\times\R : \indicator_{[0,\infty)}(f(u)-u_0)=1\},
\end{equation}
so $\F$ corresponds to the class $\CF := \{ C_f : f \in \F\}$. This
extension is sometimes called the pseudo dimension \citep[see
e.g.][]{Anthony-Bartlett-neural-network-learning,Pollard-empirical-processes}.

\begin{theorem}
  [\citealp{Vapnik-Chervonenkis-1971}]
  \label{thm:vapnik}
  Suppose that $0\leq\ell(y,y')\leq M <
  \infty$. Then,
  \begin{equation}
    \label{eq:vcbound}
    \P \left(\sup_{f\in\F} |R(f) - \hat{R}_n(f)| > \epsilon\right)
    \leq 4 G(2n,\CF) \exp \left\{ - \frac{n\epsilon^2}{\Upsilon} \right\} ~,
  \end{equation}
  where $\Upsilon$ depends only on $M$ and not $n$ or $\F$.
\end{theorem}

When the loss function is unbounded, similar results hold, but we must consider
the composition of the loss function with $f$. This leads to the set
\begin{equation}
  \label{eq:vcclass}
  C_{\ell \circ f} = \{ (u,u_0,u'_0) \in \mathbb{U}\times\R\times\R :
  \indicator_{[0,\infty)}(\ell(u_0-f(u))-u'_0)=1\},
\end{equation}
and the corresponding class $\CloF := \{ C_{\ell \circ f} : f \in
\F\}$.

The concentration result in \autoref{thm:vapnik} works well for independent
data.
However, for time series, we must be able to handle dependent data.  In
particular the length $n$ of a sample
path $Y_1,\ldots,Y_n$ exaggerates how much information it contains
relative to independent observations.  Knowing
the past allows forecasters to predict future data (at least to some degree),
so actually observing those future data points gives less information about the
underlying process than in the i.i.d.\ case. Thus, while in \autoref{thm:vapnik}
the probability of large discrepancies between empirical means and their
expectations decreases exponentially in $n$, in the dependent case, the
effective sample size may be much less than $n$, resulting in looser bounds.

\section{Time series}
\label{sec:dependence}

In moving from the i.i.d.\ setting to time-series forecasting, we need a number of
modifications to our initial setup.  Rather than observing input/output pairs
$(Y_i,X_i)$, we observe a single sequence of random variables $Y_1,\ldots,Y_n$
where each $Y_i$ takes values in $\R^p$, though we can generalize to arbitrary
metric spaces at some cost in notational clarity.  We are interested in using functions
which take past observations as inputs and predict future values of the
process.  Specifically, given data from time 1 to time $n$, we wish to predict
time $n+1$. To be clear about notation, we will use the following conventions:
$\Yij{i}{j} := (Y_i,Y_{i+1},\ldots,Y_j)$, $\Yinf := Y_{-\infty:\infty}$ is an
infinite dimensional sequence; we also have the associated joint distributions
$\Pij{i}{j}$ and $\Pinf$ and $\sigma$-fields $\Sij{i}{j} = \sigma(\Yij{i}{j})$
and $\Sinf=\sigma(\Yinf)$.

While we no longer presume i.i.d.\ data, we still need to restrict the sort of
dependent process we work with.  We first remind the reader of the notion of
(strict or strong) stationarity.
\begin{definition}[Stationarity]\label{def:stationary}
  A random sequence $\Yinf$ is stationary when all its finite-dimensional
  distributions are time-invariant: for all $t$ and all non-negative integers
  $i$ and $j$, the random vectors $\Yij{t}{t+i}$ and $\Yij{t+j}{t+i+j}$ have
  the same distribution.
\end{definition}
Stationarity does not imply that the random variables $Y_i$ are independent
across time $i$, only that the marginal distribution of $Y_i$ is constant in
time.  (And similarly for $Y_{i:i+j}$.)  We limit ourselves not just to
stationary processes, but also to ones in which widely-separated observations
are asymptotically independent.  Without this restriction, convergence of the
training error to the expected risk could occur arbitrarily slowly, and
finite-sample bounds may not exist. In fact,
\citet{Adams-Nobel-VC-classes-under-ergodic} demonstrate that for ergodic
processes, finite VC dimension is enough to give consistency, but cannot itself
provide rates.  The next definition describes the sort of serial dependence
which we entertain.
\begin{definition}[$\beta$-Mixing]
  \label{defn:beta-mix}
  Consider a stationary random sequence $\Yinf$ defined on a probability space
  $(\Omega, \Sigma, \Pinf)$. Let $\P_0$ be the restriction of $\Pinf$ to
  $\Sij{-\infty}{0}$, $\P_{a}$ be the restriction of $\Pinf$ to
  $\Sij{a}{\infty}$, and $\P_{0\otimes a}$ be the restriction of $\Pinf$ to
  $\sigma(\Yij{\infty}{0},\Yij{a}{\infty})$.  The {\em coefficient of absolute
    regularity}, or {\em $\beta$-mixing coefficient}, $\beta_a$, is given by
  \begin{equation}
    \label{eq:three}
    \beta_a := \|\P_0 \times \P_{a} - \P_{0 \otimes
      a}\|_{TV},
  \end{equation}
  where $\| \cdot \|_{TV}$ is the total variation norm. A stochastic process is
  {\em absolutely regular}, or {\em $\beta$-mixing}, if $\beta_a \rightarrow 0$
  as $a\rightarrow\infty$.
\end{definition}
This is only one of many equivalent characterizations of $\beta$-mixing
\citep[see][for others]{Bradley-strong-mixing}.  This definition makes clear that a
process is $\beta$-mixing if the joint probability of events which are widely
separated in time approaches the product of the individual probabilities, that
is that $\Yinf$ is asymptotically independent.  Many common time-series models
are known to be $\beta$-mixing, and the rates of decay are known up to constant
factors which are functions of the true parameters of the process.  Among the
processes for which such results are known are ARMA models
\citep{Mokkadem1988}, GARCH models
\citep{Carrasco-and-Chen-mixing-of-GARCH-and-stoch-vol}, and certain Markov
processes \citep[see][for an overview]{Doukhan-on-mixing}.  Additionally,
functions of $\beta$-mixing processes are $\beta$-mixing, so if $\Pinf$ could
be specified by a linear time-invariant system (see below),
state-space model, vector auto regression, or some function of a hidden
Markov model, the observed data would satisfy this condition.

Knowing $\beta_a$ would let us determine the effective sample size of a time
series $\Yin$.  In effect, having $n$ dependent-but-mixing data points is like
having $\mu<n$ independent ones.  Once we determine the correct $\mu$, we can
(as we will show) use concentration results for i.i.d.\ data like
\autoref{thm:vapnik} with small corrections.

\section{Linear time-invariant dynamical systems}
\label{sec:line-time-invar}

Our goal in this paper will be to derive risk bounds for linear
time-invariant dynamical systems (LTIs). Such models presume an
underlying latent process, and attempt, given observations,
to learn that process, predict future values of the latent process,
and forecast future observations. Learning algorithms for these goals
are linear functions of all previously observed
values. Such models nest many common time-series forecasting
techniques---ARIMA models, GARCH models, linear-Gaussian state-space
models---but, due to their (in general) dependence on the entire past,
are not covered by the work discussed in
\autoref{sec:related-work}. We present the general form of such models
here, provide a canonical forecasting algorithm, and discuss some
properties of such models.

Linear dynamical systems model observations $Y_i$ as:
\begin{align}
  \label{eq:state-space-model2}
  Y_i &= Z\alpha_i + \epsilon_i, \nonumber\\
  \alpha_{i+1} &= T\alpha_i+ G\eta_{i+1}. 
\end{align}
This is essentially a hidden markov model under certain conditions:
$Y_i$ denotes observations, $\alpha_i$ are hidden ``state'' variables, 
$\epsilon_i$ and $\eta_i$ are both absolutely continuous random noise
with with $\E[\epsilon_i] = \E[\eta_i]=0$, 
$\E[\epsilon_i\epsilon_j^\top]=\delta_{ij}H$, and $\E[\eta_i\eta_j^\top]=\delta_{ij}Q$ for
all $i,\ j$. We further assume that $\epsilon$ and $\eta$ are mutually independent
even though this is not strictly necessary, because it makes notation simpler.
We require stationarity for our results, and so we also
require the LTI to be stationary. This amounts to forcing
the complex eigenvalues of $T$ to lie inside the unit circle. We note
that the condition that the noise distributions are absolutely
continuous means that HMMs in the strictest sense are not members of
this family. We also allow the parameter
matrices $Z$, $T$, $H$, $G$, and
$Q$ to depend on a (possibly unknown)
parameter vector $\theta$, and assume that $H$ and $Q$ are positive
definite for all $\theta$.

\begin{algorithm}[t!]
  \SetKwInOut{Input}{Input}
  \DontPrintSemicolon
  Recursively generate minimum mean squared error predictions
  $\widehat{Y}_i$ using the state space model in
  (\ref{eq:state-space-model2}).\;
  \Input{Initial guesses for the mean and variance of $\alpha_1$:
    $\hat{\alpha}_1$ and $\hat{P}_1$}
  \KwSet $\widehat{Y}_1 = T\hat{\alpha}_1$.\;
  \For{$1 \leq i \leq n$}{
    Filter 
    \begin{align*}
      v_i &= Y_i - \widehat{Y}_i, & F_i &= (Z\hat{P}_iZ^\top + H)^{-1},\\
      K_i &= T\hat{P}_iZ^\top F_i, & L_i &= T - K_iZ,\\
      \hat{\alpha}_{i+1} &= T\hat{\alpha}_i + K_i v_i, & \hat{P}_{i+1}
                                        &= T\hat{P}_iL_i^\top + GQG^\top. 
    \end{align*}
    Predict  
    $\widehat{Y}_{i+1} = Z\hat{\alpha}_{i+1}.$\;
  }
  \Return{$\widehat{Y}_{1:n+1}$}
  \caption{Kalman filter\label{alg:kalman}}
\end{algorithm}

The {\em filtering} problem uses observations $\Yij{1}{i}$ up to time $i$ to learn
information about the distribution of $\alpha_i$. Then, conditional on
an estimate $\hat{\alpha}_i$, we can forecast $\hat{\alpha}_{i+1}$ and
hence derive a prediction $\hat{Y}_{i+1}$.
For models of this form, one uses the Kalman filter
\citep{Kalman,Durbin-Koopman-state-space-methods,AndersonMoore2012} 
both to estimate the latent variables, $\alpha_{1:n+1}$ and to
generate predictions $\hat{Y}_{1:n+1}$
(\autoref{alg:kalman}). This procedure gives the minimum
mean-squared error predictions of $\alpha_{i+1}$ (and hence of
$Y_{i+1}$) given $\Yij{1}{i}$ in the sense that 
\begin{equation}
  \hat{\alpha}_{i+1} = \argmin_a
  \E\left[tr\left((\alpha_{i+1}-a)(\alpha_{i+1}-a)^\top\right)\ \vert\ \Yij{1}{i}\right].
\end{equation}
Furthermore, if $\epsilon_i$ and
$\eta_i$ are Gaussian, then \autoref{alg:kalman} also gives the
likelihood for the unknown parameter vector $\theta$.
To estimate the unknown 
parameters, we either: (1) maximize the likelihood returned by the
filter; or (2) use the EM algorithm, alternating between running the Kalman
filter (the E-step) and maximizing the conditional likelihood by least squares
(the M-step).  Bayesian estimation works like EM, replacing the M-step with
Bayesian updating. 

Predictions based on \autoref{alg:kalman} are linear functions
of previous observations, but these predictions depend
on of all previous observations rather than simply a fixed
number as would be the case with, say, autoregressive models. More
specifically, $\hat{Y}_{1:n+1} = \mathbf{B}\Yij{1}{n}$ where
\begin{equation}
  \label{eq:bmatrix}
  (\mathbf{B})_{ij} = b_{ij} =
  \begin{cases} \displaystyle{Z
      \prod_{k=j+1}^{i-1} L_kK_j} & i-j>1\\ 
    ZK_j & i-j=1\\
    0 & i-j \leq 0
  \end{cases}
\end{equation}
for all $1\leq i \leq n+1$. Here $b_{ij}$ is the weight for predicting
$Y_i$ based on $Y_j$. Because of this dependence on the entire
past, we will require some information about the behavior of the
matrices $b_{ij}$ in terms of $i$ and $j$. Define $\lambda_{k}(A)$ to
be the $k^{th}$ largest absolute eigenvalue of a square matrix $A$,
and let $\lambda_{\max}(A) = \max_k |\lambda_k(A)|$ and
$\lambda_{\min}(A)=\min_k |\lambda_k(A)|$. Proofs of the
following results are given in \autoref{sec:proofs-results-srefs}.
\begin{lemma}
  \label{lem:kalmanExpDecay}
  If $H$ is positive definite and $r=\lambda_{\max}(T)<1$, then
  $\lambda_{\max}(b_{ij})=O(r^{i-j-1})$ for any $i>1$, $j
  < i$.
\end{lemma}

While $\hat{P}_i = Var(\alpha_{i+1} | Y_{1:i})$ changes with $i$, for
stationary LTIs, $\hat{P}_i$ converges to a limiting value as
$i\rightarrow\infty$. This means that the algorithm converges to a
steady state as $i$ grows. The next result gives the values to which
the algorithm converges, and shows that this convergence occurs
exponentially fast.
\begin{lemma}
  \label{lem:kalmanConv}
  If $r:=\lambda_{\max}(T)<1$ then:
  \begin{enumerate}
  \item The solution, $\overline{P}$ to the matrix equation
    $P = TPT^\top - TPZ^\top(ZPZ^\top + H)^{-1}ZPT^\top + GQG^\top$
    exists and is positive definite.
  \item $\hat{P}_i \rightarrow \overline{P}$ and $K_i\rightarrow
    \overline{K}=T\overline{P}Z^\top(Z\overline{P}Z^\top + H)^{-1}$.
  \item $0<\lambda_{\max}(T-\overline{K}Z) =: \rho<1$
  \item For any matrix norm
    $\lVert \cdot\rVert$, $\lVert \hat{P}_i - \overline{P} \rVert =
    O\left(\rho^i\right)$, $\lVert F_i - \overline{F}\rVert = O(\rho^i)$, and $\lVert K_i
    - \overline{K}\rVert = O\left(\rho^i\right)$.
  \end{enumerate}
\end{lemma}
Since the Kalman filter algorithm converges quickly to a steady
state, one could instead use \autoref{alg:ssKalman} which approximates
\autoref{alg:kalman} but is more computationally efficient. 
\begin{algorithm}[t!]
  \SetKwInOut{Input}{Input}
  \DontPrintSemicolon
  Recursively generate approximate minimum mean squared error predictions
  $\widehat{Y}_i$ using the state space model in
  (\ref{eq:state-space-model2}).\;
  \Input{Initial guess for the mean of $\alpha_1$:
    $\hat{\alpha}_1$}
  \KwSet $\widehat{Y}_1 = Z\hat{\alpha}_1$.\;
  \KwSolve $P = TPT^\top - TPZ^\top(ZPZ^\top + H)^{-1}ZPT^\top +
  GQG^\top$ for $P$, denote the solution as $\overline{P}$\;
  \KwSet $\overline{K}=T\overline{P}Z^\top (Z\overline{P}Z^\top + H)^{-1}$\;
  \For{$1 \leq i \leq n$}{
    Filter 
    \begin{align*}
      v_i &= Y_i - \widehat{Y}_i, &\hat{\alpha}_{i+1} &= T\hat{\alpha}_i + \overline{K} v_i,
    \end{align*}
    Predict  
    $\widehat{Y}_{i+1} = Z\hat{\alpha}_{i+1}.$\;
  }
  \Return{$\widehat{Y}_{1:n+1}$}
  \caption{Steady state approximate filter\label{alg:ssKalman}}
\end{algorithm}

For \autoref{alg:ssKalman}, we can similarly write predictions as
linear functions of previous observations. In this case,
$\hat{Y}_{1:n+1} = \mathbf{S}\Yin$ where the prediction weights are given by
\begin{equation}
  \label{eq:smatrix}
  (\mathbf{S})_{ij} = s_{ij} = s_{i-j-1} =
  \begin{cases} Z
      (T-\overline{K}Z)^{i-j-1}\overline{K} & i-j>0\\ 
    0 & i-j\leq 0
  \end{cases}
\end{equation}
Notice in particular that the weights depend only on the difference
$i-j$ for this algorithm. The next
result shows that the prediction weights for the two algorithms
converge rapidly.

\begin{lemma}
\label{lem:filterWeightsConv}
  If $\lambda_{\max}(T)<1$ and $H$ positive
  definite. Then, $\lVert s_{ij} \rVert = O(\rho^{i-j-1})$ and for $j<i$,
  $
  \lVert b_{ij} - s_{ij} \rVert = O(\rho^j).
  $
\end{lemma}

We will refer to the class of predictors given by \autoref{alg:kalman}
as $\F_1$ and those given by \autoref{alg:ssKalman} as $\F_2$.

\section{Risk bounds}
\label{sec:risk-bounds}

With the relevant background in place, we can put the pieces together to derive
our results. We use $\beta$-mixing to find out how much information is in the
data and VC dimension to measure the capacity of the state-space model's
prediction functions. The result is a bound on the generalization error of the
chosen function $\hat{f}$.  After slightly modifying the definition of ``risk''
to fit the time-series forecasting scenario and stating necessary technical
assumptions, we derive risk bounds for traditional time-series forecasting models.

\subsection{Setup and assumptions}
\label{sec:setup-assumptions}

We observe a finite subsequence of random vectors $\Yin$ from a process $\Yinf$
defined on a probability space $(\Omega, \Sigma, \P_\infty)$, with $Y_i \in
\R^p$. We make the following assumption on the process.
\begin{assumption}
  \label{ass:A}
  $\P_\infty$ is a stationary, $\beta$-mixing process with mixing
  coefficients $\beta_a$, $\forall a>0$.\footnote{In order to apply the results, one must either know
    $\beta_a$ for some $a$ or be able to estimate it with sufficient precision
    and accuracy. \citet{estimating-beta-mixing,McDonaldShalizi2015}
    show how to estimate the 
    mixing coefficients non-parametrically, based on a single sample from the
    process. However, those results (and
  those contained in this paper) only apply if the data generating
  process {\em is} $\beta$-mixing, an assumption that cannot be verified.}
\end{assumption}

Under stationarity, the marginal distribution of $Y_t$ is the same for all $t$.
We deal mainly with the joint distribution of $\Yini$, where we observe the
first $n$ observations and try predicting $Y_{n+1}$.  For the
remainder of this
paper, we will call this joint distribution $\P$. Our results extend to
predicting more than one step ahead, but the notation becomes cumbersome.

We must define generalization error and training error slightly differently for
time series than in the i.i.d.\ setting.  Using the same notion of loss functions
as before, we consider prediction functions $f : \R^{n\times p} \rightarrow
\R^p$.
The function $f$ may use some or all of the past to
generate predictions.  A function using only the most recent $d$ observations
as inputs will be said to have \emph{fixed memory} of length $d$. Other
functions, in particular, the linear time-invariant systems we discuss
below, have \emph{growing memory} which means that $f$ may use all the
previous data to predict the next data point.
These concepts require us to state with some care what we mean by
prediction functions, and by time-series training error and risk.

\begin{definition}
  [Time-series risk]
  \begin{equation}
  R_n(f) := \E \Big[ \mnorm{ Y_{n+1} - f(\Yin)} \Big].
  \end{equation}
\end{definition}
The expectation is taken with respect to the joint distribution $\P$ and
therefore depends on $n$.

\begin{definition}
  [Time-series training error]
  \begin{equation}
  \hat{R}_n(f) := \frac{1}{n-d} \sum_{i=d+1}^{n} \mnorm{Y_{i} - f(\Yij{1}{i-1})}.
  \end{equation}
\end{definition}
In order to make use of this single definition of training error, we let $d
\geq 0$.  In fixed memory cases---say an AR(2)---$d$ has an obvious meaning,
and $f(\Yij{1}{i})=f(\Yij{i-d+1}{i})$ by definition of fixed memory, while with
growing memory, $d=0$, and we define $\Yij{1}{0}:=\varnothing$.

To control the generalization error for time-series forecasting, we make one
final assumption, about the possible magnitude of the losses.  Specifically, we
weaken the bounded loss assumption we used in \autoref{sec:learning-theory} to
allow for unbounded loss as long as we retain some control on moments of the
loss.
\begin{assumption}
  \label{ass:B}
  Assume that for all $f \in \F$ and all $d \in \mathbb{N}$
  \begin{equation}
    Q_d(f) := \sqrt{\E_\P\Big[\mnorm{Y_{d+1}-f(\Yij{1}{d})}^2
      \Big]}\leq M < \infty.
  \end{equation}
\end{assumption}

\autoref{ass:B} will be satisfied for $Y_i$ sub-Gaussian, as well as other
distributions with bounded second moment.  These include, for instance,
heavy-tailed L{\'e}vy noises where the tails of the pdf decay faster than an
inverse cubic. 

\subsection{Fixed memory}
\label{sec:fixed-memory}

We can now state our results giving finite sample risk bounds for the problem
of time-series forecasting.  We begin with the fixed memory setting; the next
section will allow the memory length to grow.

\begin{theorem}
  \label{thm:bound1}
  Suppose that \autoref{ass:A} and \autoref{ass:B} hold, that the model class
  $\F$ has a fixed memory length $d<n$, and that we have a sample
  $\Yin$. Let $\mu$ and $a>d$ be integers such that $2\mu a+d=
  n$.\footnote{By making appropriate modifications to the definition
    of the training error and some of the proof elements, one could
    allow $2\mu a +d <n$, but we obviate this issue for the sake of
    clarity.} Then, for all $0<\epsilon<\frac{e^{3/2}}{\sqrt{2}}$,
  \begin{align}
    \label{eq:bound1}
        \P\left( \sup_{f \in \F} \frac{R_n(f) - \widehat{R}_n(f)}{Q_d(f)}
        > \epsilon \right)
      \leq 8 G(n, \CloF)\exp\left\{ -
        \frac{\mu\exp\left(W\left(-\frac{2\epsilon^2}{e^4}\right)+4\right)}{4}
      \right\} + 2\mu\beta_{a-d},
    \end{align}
    where $W(\cdot)$ is the Lambert W function.
\end{theorem}

The implications of this theorem are considerable. Given a finite sample of
length $n$, we can say that with high probability, future prediction errors
will not be much larger than our observed training errors. It makes no
difference whether the model is correctly specified. This stands in stark
contrast to model selection tools like AIC or BIC which appeal to
asymptotics. Moreover, given a model class $\F$, we can say exactly how much
data we need to have good control of the prediction risk. As the effective data
size increases, the training error is a better and better estimate of the
generalization error, uniformly over all of $\F$, provided
$\mu,a\rightarrow\infty$ such that $\beta_{a-d}=o(1/\mu)$.

The Lambert W function in the exponential term deserves some explanation. The
Lambert W function is defined as the inverse of $f(w) = w\exp{w}$
\citep[c.f.][]{Corless-et-al-on-Lambert-W}. A strictly, but only slightly,
worse bound can be achieved by noting that
\begin{equation}
\label{eq:worsebound}
\exp\left(W\left(-\frac{2\epsilon^2}{e^4}\right)+4\right)\leq \frac{\epsilon^{8/3}}{4^{2/3}}
\end{equation}
for all $\epsilon\in [0,1]$ \citep[see][for the
derivation]{Cortes-Mansour-Mohri-bounds}.

The difference between expected and empirical risk is only interesting when
$R_n(f)$ exceeds $\widehat{R}_n(f)$.  Due to the supremum preceding $\frac{R_n(f) - \widehat{R}_n(f)}{Q_d(f)}$, events where the
training error exceeds the expected risk are irrelevant (as this term will
be negative). Therefore, we are only
concerned with $0 \leq \hat{R}_n(f) \leq R_n(f)$. Of course, as discussed in
\autoref{sec:learning-theory}, for most estimation procedures, $f$ is chosen to
make $\widehat{R}_n(f)$ as small as possible.

Before we prove \autoref{thm:bound1}, we will state a corollary which puts the
same result in a form that is sometimes easier to use.

\begin{corollary}
  \label{cor:bound1a}
  Under the conditions of \autoref{thm:bound1}, for any $f \in \F$, the
  following bound holds with probability at least $1-\eta$, for all
  $\eta>2\mu\beta_{a-d}$:
  \begin{equation}
    \label{eq:bound-with-known-vc}
    R_n(f) \leq \widehat{R}_n(f) +Me^2
    \sqrt{ \frac{ \mathcal{E}(4-\log \mathcal{E})}{2}},
  \end{equation}
  with
  \begin{equation}
    \label{eq:bound-magnitude}
    \mathcal{E} = \frac{4\log G(n,\CloF) + 4\log 8/\eta'}{\mu},
  \end{equation}
  and $\eta'=\eta-2\mu\beta_{a-d}$.
\end{corollary}
We now prove both \autoref{thm:bound1} and \autoref{cor:bound1a} to provide the
reader with some intuition for the types of arguments necessary.  We defer
proof of the remainder of our results to
\autoref{sec:proofs-results-srefs}.

\begin{proof}[of \autoref{thm:bound1} and \autoref{cor:bound1a}]
  The first step is to move from the actual sample size $n$ to the effective
  sample size $\mu$ which depends on the $\beta$-mixing behavior.  Let $a$ and
  $\mu$ be non-negative integers such that $2a\mu + d = n$. Now divide
  $\Yij{d+1}{n}$ into $2\mu$ blocks, each of length $a$, ignoring the
  first $d$ observations. Identify the
  blocks as follows:
  \begin{align}
    U_j &= \{Y_i: 2(j-1)a + d+1 \leq i \leq (2j-1)a+d\},\\
    V_j &= \{Y_i : (2j-1)a + d+1\leq i \leq 2ja+d\}.
  \end{align}
  Let $\mathbf{U}$ be the sequence of odd blocks $U_j$, and let $\mathbf{V}$ be
  the sequence of even blocks $V_j$. A graphical depiction of the
  blocking procedure is shown in \autoref{fig:blocking}.
  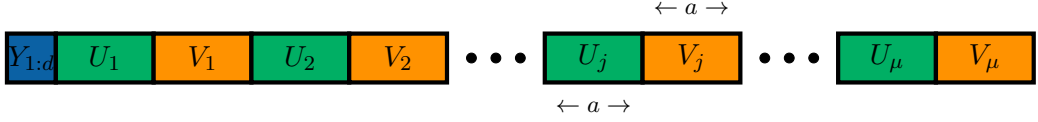
\begin{figure}
    \centering
    \begin{tikzpicture}[scale=.65]
      \draw [fill=bluemain,ultra thick] (-1,0) rectangle (0,1);
      \draw [fill=greenmain,ultra thick] (0,0) rectangle (2,1);
      \draw [fill=orangemain,ultra thick] (2,0) rectangle (4,1);
      \draw [fill=greenmain,ultra thick] (4,0) rectangle (6,1);
      \draw [fill=orangemain,ultra thick] (6,0) rectangle (8,1);
      \draw [fill] (8.5,.5) circle [radius=.1];
      \draw [fill] (9,.5) circle [radius=.1];
      \draw [fill] (9.5,.5) circle [radius=.1];
      \draw [fill=greenmain,ultra thick] (10,0) rectangle (12,1);
      \draw [fill=orangemain,ultra thick] (12,0) rectangle (14,1);
      \draw [fill] (14.5,.5) circle [radius=.1];
      \draw [fill] (15,.5) circle [radius=.1];
      \draw [fill] (15.5,.5) circle [radius=.1];
      \draw [fill=greenmain,ultra thick] (16,0) rectangle (18,1);
      \draw [fill=orangemain,ultra thick] (18,0) rectangle (20,1);
      \node at (-.5,.5) {$Y_{1:d}$};
      \node at (1,.5) {$U_1$};
      \node at (3,.5) {$V_1$};
      \node at (5,.5) {$U_2$};
      \node at (7,.5) {$V_2$};
      \node at (11,.5) {$U_j$};
      \node at (13,.5) {$V_j$};
      \node at (17,.5) {$U_\mu$};
      \node at (19,.5) {$V_\mu$};
      \footnotesize
      \node at (11,-.5) {$\leftarrow a \rightarrow$};
      \node at (13,1.5) {$\leftarrow a \rightarrow$};
    \end{tikzpicture}
    \caption{The blocking procedure divides
      $\Yij{d+1}{n}$ into $2\mu$ alternating blocks $U_j$
      (green) and $V_j$ (orange)
      each of length $a$. It ignores the
      first $d$ observations $\Yij{1}{d}$ (blue).}
    \label{fig:blocking}
  \end{figure}
  Finally, let $\mathbf{U}'$ be a sequence of blocks which are mutually
  independent but such that each block has the same distribution as a block
  from the original sequence. That is construct $U'_j$ such that
  \begin{align}
    \mathcal{L}(U'_j) = \mathcal{L}(U_j) = \mathcal{L}(U_1),
  \end{align}
  where $\mathcal{L}(\cdot)$ means the probability law of the argument.

  Let $\widehat{R}_\mathbf{U}(f)$, $\widehat{R}_\mathbf{U'}(f)$, and
  $\widehat{R}_\mathbf{V}(f)$ be the empirical risk of $f$ based on the block
  sequences $\mathbf{U}$, $\mathbf{U'}$, and $\mathbf{V}$ respectively. We have
  \begin{align}
    \widehat{R}_n(f) &= \frac{1}{n-d}\sum_{i=d+1}^{n} \mnorm{Y_{i} -
      f(\Yij{1}{i-1})}\\
    &= \frac{1}{n-d} \left[\sum_{i:Y_{i} \in \mathbf{U}}\mnorm{Y_{i} -
        f(\Yij{1}{i-1})}+\sum_{j:Y_{j} \in \mathbf{V}}\mnorm{Y_{j} -
        f(\Yij{1}{j-1})}\right]\\
    &= \frac{1}{2} \left[\frac{2}{n-d}\sum_{i:Y_{i} \in \mathbf{U}}\mnorm{Y_{i} -
        f(\Yij{1}{i-1})}+\frac{2}{n-d}\sum_{j:Y_{j} \in \mathbf{V}}\mnorm{Y_{j} -
        f(\Yij{1}{j-1})}\right]\\
    &= \frac{1}{2}\left[\widehat{R}_\mathbf{U}(f)+\widehat{R}_\mathbf{V}(f)\right].
  \end{align}

  Then,
  \begin{align}
    \lefteqn{\P\left( \sup_{f \in \F} \frac{R_n(f) - \widehat{R}_n(f)}{Q_d(f))} >
        \epsilon \right)}\\ &= \P\left( \sup_{f \in \F} \left[\frac{R_n(f)
          - \widehat{R}_\mathbf{U}(f)}{2 Q_d(f)} + \frac{R_n(f) -
          \widehat{R}_\mathbf{V}(f)}{2 Q_d(f)} \right]
      > \epsilon \right)\nonumber\\
    &\leq \P \left( \sup_{f \in \F} \frac{R_n(f) -
        \widehat{R}_\mathbf{U}(f)} {Q_d(f)}  + \sup_{f \in \F}
      \frac{R_n(f) - \widehat{R}_\mathbf{V}(f)} {Q_d(f)} >
      2\epsilon \right) \label{eq:first}\\
    & \leq \P \left( \sup_{f \in \F} \frac{R_n(f) -
        \widehat{R}_\mathbf{U}(f)} {Q_d(f)} > \epsilon \right) +
    \P \left( \sup_{f \in \F} \frac{R_n(f) -
        \widehat{R}_\mathbf{V}(f)} {Q_d(f)} > \epsilon \right) \label{eq:second}\\
    &= 2 \P \left( \sup_{f \in \F} \frac{R_n(f) -
        \widehat{R}_\mathbf{U}(f)} {Q_d(f)} > \epsilon
    \right)\label{eq:third}\\
    &= 2 \P \left( \sup_{f \in \F} \frac{R_\dplus(f) -
        \widehat{R}_\mathbf{U}(f)} {Q_d(f)} > \epsilon \right).\label{eq:fourth}
  \end{align}
  Here, \autoref{eq:first} follows by the convexity of the supremum and
  \autoref{eq:second} by a union bound.  Now, for \autoref{eq:fourth}, as $f$
  has fixed memory $d$, we have that $f(\Yij{1}{j}) =
  f(\Yij{j-d+1}{j})$. Thus
  \begin{align*}
  R_n(f) &= \E \big[ \mnorm{Y_{n+1}-f(\Yij{1}{n})} \big] = \E \big[
  \mnorm{Y_{n+1}-f(\Yij{n-d+1}{n})} \big]\\ &= \E \big[
  \mnorm{Y_{d+1}-f(\Yij{1}{d})} \big]=R_\dplus(f)
  \end{align*}
  by stationarity and, similarly, $ Q_d(f) = \sqrt{\E \big[
    \mnorm{Y_{d+1}-f(\Yij{1}{d})} \big]}$. Thus $R_\dplus(f)$ and $Q_d(f)$ depend on
  only $d+1$ data values. Likewise, a ``point'' in the training error
  $\widehat{R}_\mathbf{U}(f)$ depends on $d+1$ data values. Therefore, a
  prediction at any $Y_i$ in some block $U_j$ is separated by at least $a-d$
  observations from any $Y_{i'}$ in different block $U_{j'}.$ Furthermore, we
  are estimating an expectation $R_\dplus(f)$ which depends on $d+1$ values with an
  empirical expectation $\widehat{R}_\mathbf{U}(f)$ which is a dependent sum of
  components each of which depends on $d+1$ observed data values.  Therefore,
  we can apply Lemma 4.1 in \citet{Yu-rates-of-convergence} (reproduced as
  \autoref{lem:yu} in \autoref{sec:lemmas}) to the event $\left\{\sup_{f \in
      \F} \frac{R_\dplus(f) - \widehat{R}_\mathbf{U}(f)} {Q_d(f)} >
    \epsilon\right\}$.  This allows us to move from statements about the
  dependent blocks in $\widehat{R}_\mathbf{U}(f)$ to statements about the
  independent blocks in $\widehat{R}_\mathbf{U'}(f)$ with a slight correction
  which accounts for the worst-case dependence between adjacent blocks:
  $\beta_{a-d}$. Therefore,
  \begin{align}
      2 \P \left( \sup_{f \in \F} \frac{R_\dplus(f) -
          \widehat{R}_\mathbf{U}(f)} {Q_d(f)} > \epsilon \right)
    \leq  2 \P \left( \sup_{f \in \F} \frac{R_\dplus(f) -
        \widehat{R}_\mathbf{U'}(f)} {Q_d(f)} > \epsilon \right) + 2\mu\beta_{a-d},\label{eq:indepprob}
  \end{align}
  where the probability on the right of \autoref{eq:indepprob} is for the
  $\sigma$-field generated by the independent block sequence $\mathbf{U}'$.
  Let us now introduce the growth function for the class of ``blocked''
  predictor losses $\mathcal{C}^\Sigma_{\ell \circ \F} = \{C^\Sigma_{\ell \circ \F}:
  f \in \F\}$, where\footnote{The $\Sigma$ indicates a sum over the coordinates of the block.}
  \begin{equation}
    C^\Sigma_{\ell \circ \F} = \left\{ (\mathbf{u},\mathbf{d},b) :
  \indicator_{[0,\infty)}\left(\sum_{j=1}^a \ell(d_j-f(u_j))-b\right)=1,\ \
  \mathbf{u}\in\mathbb{U}^a,\ \ \mathbf{d}\in \R^{k\times a},\ b\in\R\right\}.
  \end{equation}
  This allows us to state that
  \begin{align}
    \label{eq:almost-finished-bound}
    2\P \left( \sup_{f \in \F} \frac{R_\dplus(f) -
          \widehat{R}_\mathbf{U'}(f)} {Q_d(f)} > \epsilon
      \right)
    & \leq 8 G(2\mu, \mathcal{C}^\Sigma_{\ell \circ \F} ) \exp\left\{ -
      \frac{\mu\exp\left(W\left(-\frac{2\epsilon^2}{e^4}\right)+4\right)}{4} \right\}
  \end{align}
  where we have applied \autoref{cor:corteseq} to bound the independent blocks
  $\mathbf{U}'$. Since there are $\mu$ independent blocks, this upper bound is
  in terms of $\mu$ effectively independent data ``points'' penalized by the
  correction $\beta_{a-d}$ which adjusts for the worst case dependence between
  ``points'' rather than $n$ dependent data points. However, the growth
  function is in terms of the class of blocked predictors and depends on
  $a$.  To remove this dependence, we present a novel result,
  \autoref{lem:vcequiv}. This lemma shows that $G(2\mu, C^\Sigma_{\ell \circ
    \F})\leq G(2\mu a, \CloF) \leq G(n, \CloF)$, giving the desired result.

  To prove the corollary, set the right hand side of
  \eqref{almost-finished-bound} to $\eta$, take $\eta'=\eta -
  2\mu\beta_{a-d}$, and solve for $\epsilon$.  We get that with
  probability at least $1-\eta$, for all
  $f\in\F$,
  \begin{equation}
    \frac{R_n(f) - \widehat{R}_n(f)}{Q_d(f)} \leq \epsilon.
  \end{equation}
  Solving the equation
  \begin{align}
    \eta' &= 8 G(n, \CloF)\exp\left\{ -
      \frac{\mu\exp\left(W\left(-\frac{2\epsilon^2}{e^4}\right)+4\right)}{4} \right\}\\
    \intertext{implies} \epsilon &= e^2\sqrt{ \frac{
        \mathcal{E}(4-\log \mathcal{E})}{2}}\\
    \intertext{with} \mathcal{E} &= \frac{4 \log G(n, \CloF) + 4\log 8/\eta'}{\mu}.
  \end{align}
\end{proof}

The only obstacle to the use of \autoref{thm:bound1} is knowledge of the
$G(n,\CloF)$ or $\vcd(\CloF)$. For some models, these can be calculated
explicitly.
\begin{theorem}
  \label{thm:vcd-ar}
  If $\ell(y-y')$ is monotone increasing in $|y-y'|$, then for the class of
  AR($d$) models,
  \begin{align}
    G(n,\CloF) \leq \left(\frac{2en}{d+1}\right)^{d+1}.
  \end{align}
  If $\ell(y-y')$ is monotone increasing in each coordinate of $|y_j-y_j'|$, then for
  the class of vector autoregressive models of order $d$ with $p$ time-series,
  \begin{align}
    G(n,\CloF) \leq
    \left(\frac{2en}{d+1}\right)^{p(d+1)}.
  \end{align}
\end{theorem}

The proof of \autoref{thm:vcd-ar} is given in
\autoref{sec:proofs-results-srefs}. We note that this result applies equally to
Bayesian VARs, however, this is likely conservative as the prior tends to
restrict the effective complexity of the function class.\footnote{Here we
  should mention that these risk bounds are frequentist in nature. We mean that
  if one treats Bayesian methods as a regularization technique and predicts
  with the posterior mean or mode, then our results hold.  However, from a
  subjective Bayesian perspective, our results add nothing since all inference
  can be derived from the posterior.  For further discussion of the frequentist
  risk properties of Bayesian methods under mis-specification, see for example
  \citet{Kleijn-van-der-Vaart,Muller-risk-of-Bayesian-inference} or
  \citet{CRS-dynamics-of-bayes}.}

\subsection{Bounds for LTIs}
\label{sec:growing-memory}

As discussed in~\autoref{sec:line-time-invar}, our goal is to derive
similar bounds for linear time-invariant
dynamical systems which produce forecasts via \autoref{alg:kalman} or
\autoref{alg:ssKalman}. We begin with a result for the simple case
with predictions generated by \autoref{alg:ssKalman}.

\begin{theorem}
  \label{thm:ssKalman}
  Suppose that \autoref{ass:A} and \autoref{ass:B} hold, and that the model
  class $\F_2$ is generated by \autoref{alg:ssKalman} with $\lambda_{\max}(T)<1$.
  Further assume that
  the loss function $\ell$ is a norm and let $\ell^*(A) = \sup_{z
    \neq 0} \frac{\ell(Az)}{\ell(z)}$ be the matrix norm induced by $\ell$.
  Given a time-series of length $n$, fix some $1\leq d < n$, and let $\mu$ and $a$
  be integers such that $2\mu a+d = n$.  Then
  \begin{align}
    \lefteqn{\P\left(\sup_{f_2\in\F_2}\frac{R_n(f_2) -
          \hat{R}_n(f_2)-\delta_d(f_2)}{Q_d(f_2)} > \epsilon\right)} \nonumber\\ & \leq
    8 \GF(n, \mathcal{C}_{\ell \circ \F'_2}) \exp\left\{ -
        \frac{\mu\exp\left(W\left(-\frac{2\epsilon^2}{e^4}\right)+4\right)}{4}
      \right\} + 2\mu\beta_{a-d},\\
    &\leq
    8\left(\frac{2en}{d+1}\right)^{p(d+1)} \exp\left\{ -
        \frac{\mu\exp\left(W\left(-\frac{2\epsilon^2}{e^4}\right)+4\right)}{4}
      \right\} + 2\mu\beta_{a-d},\nonumber
  \end{align}
  where
  \begin{equation}
    \delta_d(f_2) = \frac{1}{n-d}\sum_{i=d+1}^{n}\left[ \mnorm{Y_{i} - \sum_{j=i-d}^{i-1}
      s_{ij}Y_j} - \mnorm{Y_{i} - \sum_{j=1}^{i-1}
      s_{ij}Y_j}\right]+\E[ \mnorm{Y_1 }]\sum_{j=1}^{n-d}\ell^*(s_{nj}).
  \end{equation}
\end{theorem}

The $\delta_d(f_2)$ term arises from taking a fixed-memory approximation, of
length $d$, to predictors with growing memory\footnote{There are several ways
  one could make such an approximation.  To simplify the proof, we have
  simplify set all coefficients at lags beyond $d$ to zero rather
  than, e.g., asking for the $d$-memory linear 
  predictor coming closest in $L_2$ to the infinite-memory predictor.}.  As
becomes clear in the proof (see \autoref{sec:proofs-results-srefs}), we make
this approximation to apply the previous theorem, but it involves a trade-off.
As $d\nearrow n$, $\delta_d(f_2) \searrow 0$, but this drives $\mu \searrow 0$,
resulting in fewer effective training points, whereas smaller $d$ has the
opposite effect.  The summation inside square braces on the left is the
difference between empirical risks for $f_2$ and that of the truncated predictor
$f_2^{\prime}$ which uses only the most recent $d$ data values. That is,
\begin{equation}
  \frac{1}{n-d-1}\sum_{i=d}^{n-1}\left[ \mnorm{Y_{i+1} - \sum_{j=i-d+1}^i
      s_{ij}Y_j} - \mnorm{Y_{i+1} - \sum_{j=1}^i
      s_{ij}Y_j}\right] = \hat{R}_n(f_2^{\prime}) - \hat{R}_n(f_2).
\end{equation}
This term is easily calculated from the data. Also, $\delta_d(f_2)$ depends on
$\E\big[\mnorm{Y_1}\big]$ which is not necessarily desirable. However,
\autoref{ass:B} has the consequence that $\E\big[\mnorm{Y_1}\big] \leq M
<\infty$ as long as $\F_2$ allows $s_{ij}=0$. Finally, we reiterate that
$s_{ij} = s_{i-j-1}$ is a function only of the difference between $i$
and $j$. Finally, note that the upper bound depends on the growth
function of the truncated class $\F'_2$, which can be bounded using \autoref{thm:vcd-ar}.

\begin{corollary}
  \label{cor:ssKalman}
  Under the conditions of either \autoref{thm:ssKalman}, for any $f_2\in\F_2$, with probability at least $1-\eta$,
  \begin{align}
    R_n(f_2) &\leq \hat{R}_n(f_2)+\delta_d(f_2) + M e^2
    \sqrt{ \frac{ \mathcal{E}(4-\log \mathcal{E})}{2}},
  \end{align}
  where
  \begin{equation}
    \mathcal{E} = \frac{4\log G(n,\mathcal{C}_{\ell\circ \F'_2}) + 4\log 8/\eta'}{\mu},
  \end{equation}
  and $\eta'=\eta-2\mu\beta_{a-d}$.
  Furthermore, 
  \begin{equation}
    \delta_d(f_2) = O_\P(\rho^{d}).
  \end{equation}
\end{corollary}

For the case of \autoref{alg:kalman}, \autoref{ass:B} no longer makes
sense as it fails to reflect the fact each $f_1\in\F_1$ is a
time varying function of previous observations as illustrated in
\autoref{eq:bmatrix}. We therefore strengthen the assumption
slightly. Notice in particular that $f'_1$ depends on $n$ as
illustrated in \autoref{eq:bmatrix}.

\begin{assumption}
  \label{ass:C}
  For all $f'_1$ generated by truncating \autoref{alg:kalman} to depend on the most recent $d$
  observations, and all $n>d$,
  \begin{enumerate}
  \item 
    \[ 
    Q_d(f'_1) := \sqrt{\E_\P\Big[\mnorm{Y_{n+1}-f'_1(\Yij{n-d+1}{n})}^2
      \Big]}\leq M < \infty. 
    \]
  \item For every $\epsilon>0$ there exists $\delta>0$ such that, for
    all measurable sets $A$, with $\P(A)\leq \delta$, and every $f'_1$, 
    \[
    \int_A d\P \Big[\mnorm{Y_{n+1}-f'_1(\Yij{n-d+1}{n})}^2
    \Big] \leq \epsilon.
    \]
  \end{enumerate}
\end{assumption}

The first part of this assumption is analogous to \autoref{ass:B} but
for predictors $f_1\in\F_1$. However, we need only impose the
assumption on a fixed-memory version of these growing memory
predictors. Note also that for the case of \autoref{alg:ssKalman}, as
in \autoref{thm:ssKalman}, for
any $n$, $f_2\in\F_2$, we have $Q_d(f'_2) =
\sqrt{\E_\P\Big[\mnorm{Y_{n}-f'_2(\Yij{n-d}{n-1})}^2\Big]}$. For
\autoref{alg:kalman}, this is no longer the case (as $f'_1$ depends on
$n$). However, the dependence on
$n$ will not be necessary. The second part of this assumption ensures
that the loss of the truncated versions of these predictors is
uniformly integrable. Under this assumption, since predictions 
$f'_1(\Yij{n-d+1}{n})-f'_2(\Yij{n-d+1}{n}) = O_\P(1)$, 
we have that
$\lim_{n\rightarrow}Q_d(f'_1)=Q_d(f_2)$. This statement is made rigorous
in the proof.


\begin{theorem}
  \label{thm:bound3}
  Suppose that \autoref{ass:A} and \autoref{ass:C} hold, and that the model
  class $\F_1$ is generated by \autoref{alg:kalman} with $\lambda_{\max}(T)<1$.
  Further assume that
  the loss function $\ell$ is a norm and let $\ell^*(A) = \sup_{z
    \neq 0} \frac{\ell(Az)}{\ell(z)}$ be the matrix norm induced by $\ell$.
  Given a time-series of length $n$, fix some $1\leq d < n$, and let $\mu$ and $a$
  be integers such that $2\mu a+d = n$.  Then
  \begin{align}
    \lefteqn{\P\left(\sup_{f_n \in\F}\frac{R_n(f_1) -
          \hat{R}_n(f_1)-\delta_d(f_1)}{\lim_{n\rightarrow\infty}Q_d(f'_1)} > \epsilon\right)} \nonumber\\ & \leq
    8 \GF(n, \mathcal{C}_{\ell \circ \F'_2}) \exp\left\{ -
        \frac{\mu\exp\left(W\left(-\frac{2\epsilon^2}{e^4}\right)+4\right)}{4}
      \right\} + 2\mu\beta_{a-d},\\
    &\leq
    8\left(\frac{2en}{d+1}\right)^{p(d+1)} \exp\left\{ -
        \frac{\mu\exp\left(W\left(-\frac{2\epsilon^2}{e^4}\right)+4\right)}{4}
      \right\} + 2\mu\beta_{a-d},\nonumber
  \end{align}
  where
  \begin{equation}
    \delta_d(f_1) = \frac{1}{n-d}\sum_{i=d+1}^{n}\left[ \mnorm{Y_{i} - \sum_{j=i-d}^{i-1}
      b_{ij}Y_j} - \mnorm{Y_{i} - \sum_{j=1}^{i-1}
      s_{i-j-1}Y_j}\right]+\E[ \mnorm{Y_1 }]\sum_{j=1}^{n-d}\ell^*(b_{nj}).
  \end{equation}
\end{theorem}
For this result, the summation inside square braces on the left is the
difference between empirical risk for $f_1$ and that of the truncated
steady-state predictor
$f'_2$ which uses only the most recent $d$ data values. That is,
\begin{equation}
  \frac{1}{n-d}\sum_{i=d+1}^{n}\left[ \mnorm{Y_{i} - \sum_{j=i-d}^{i-1}
      b_{ij}Y_j} - \mnorm{Y_{i+1} - \sum_{j=1}^i
      s_{ij}Y_j}\right] = \hat{R}_n(f_1) - \hat{R}_n(f'_2)
\end{equation}
where $f'_2$ is given by \autoref{alg:ssKalman}. Again,
this term is easily calculated from the data. Furthermore, the upper
bound depends on the growth function of the level sets of $\F'_2$
which is easily bounded using~\autoref{thm:vcd-ar}.

\begin{corollary}
  \label{cor:bound3}
  Under the conditions of 
  \autoref{thm:bound3}, for any $f_1\in\F_1$, with probability at least $1-\eta$,
  \begin{align}
    R_n(f_1) &\leq \hat{R}_n(f_1)+\delta_d(f_1) + M e^2
    \sqrt{ \frac{ \mathcal{E}(4-\log \mathcal{E})}{2}},
  \end{align}
  where
  \begin{equation}
    \mathcal{E} = \frac{4\log G(n,\mathcal{C}_{\ell\circ \F'_2}) + 4\log 8/\eta'}{\mu},
  \end{equation}
  and $\eta'=\eta-2\mu\beta_{a-d}$.
  Furthermore, 
  \begin{equation}
    \delta_d(f_1) = O_\P(r^d) + O(\rho^{n-d+1}) + O_\P((n-d)^{-1}).
  \end{equation}
\end{corollary}

This corollary shows that $\delta_d(f_1)$ decays rapidly as long as
$d\rightarrow\infty$ and $d/n\rightarrow 0$. But for LTIs, it is simple to
compute $\delta_d(f_1)$ or $\delta_d(f_2)$ using results from \autoref{alg:kalman} or
\autoref{alg:ssKalman}, so this and \autoref{cor:ssKalman}
let us compute risk bounds for common time-series forecasting
models without appealing to the asymptotic form of $\delta_d(f)$.

\section{Bounds in practice}
\label{sec:examples}

We now show how the theorems of the previous section can be used both to
quantify prediction risk and to select models.  We first estimate a simple
stochastic volatility model using IBM return data and calculate the bound for
the predicted volatility using \autoref{cor:bound3}. Then we show how the same
methods can be used for typical macroeconomic forecasting models.

\subsection{Stochastic volatility model}
\label{sec.rb:bounds-practice}

We estimate a standard stochastic volatility model using daily log returns for
IBM from January 1962 until October 2011 ($n=12541$
observations). \autoref{fig:returns} shows the squared log-return series.

The model we investigate is
\begin{align}
  Y_i &= \sigma z_i\exp(\alpha_i/2), & z_i
  &\sim \mbox{N}(0,1),\\
  \alpha_{i+1} &= \phi \alpha_i + \eta_i, & \eta_i &\sim \mbox{N}(0,\sigma_\eta^2),\notag
\end{align}
where the disturbances $z_i$ and $w_i$ are mutually and serially independent.
Following \citet{Harvey-Ruiz-multivariate-stoch-var}, we linearize this
by taking the natural log of the observation equation:
\begin{align}
  \log Y_i^2 &= \varsigma + \frac{1}{2}\alpha_i +\epsilon_i,&
  \epsilon_i &=\log z_i^2-\xi\\
  \varsigma &= \log \sigma^2 + \xi &
  \xi &= \E[\log z_1^2]\notag.
\end{align}
The noise term $\epsilon_i$ is no longer Gaussian (it has a shifted
log-gamma distribution), but the Kalman filter will still
give the minimum-mean-squared-error linear estimate of the variance sequence
$\alpha_{1:n+1}$.  The observation variance is now $\pi^2/2$.

\begin{figure}[t!]
  \centering
  \includegraphics[width=\textwidth]{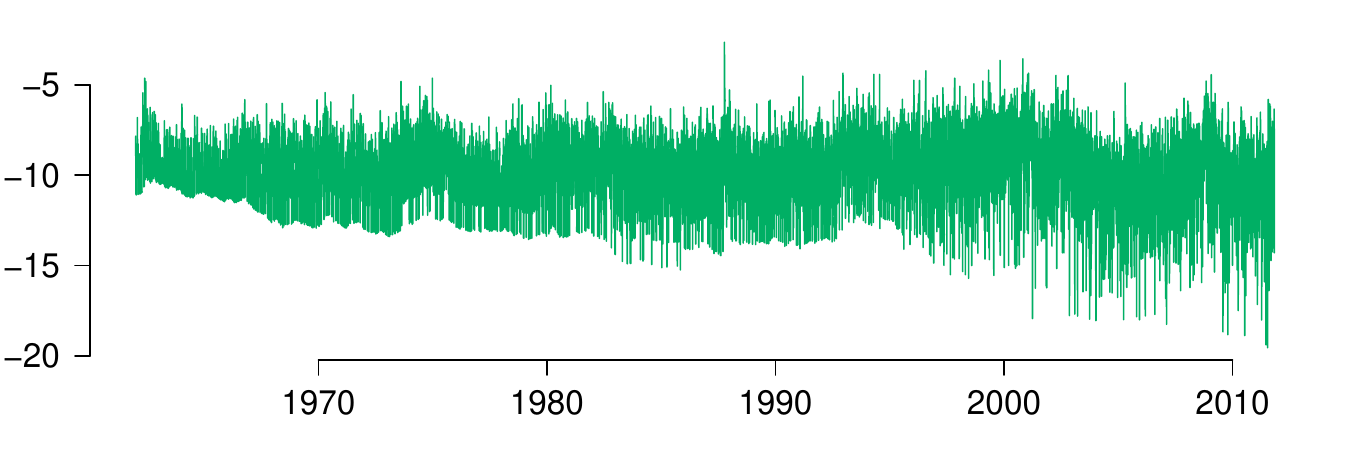}
  \caption{Daily realized volatility (squared log returns) for IBM from
    1962--2011.}
  \label{fig:returns}
\end{figure}

To match the data to the model, let $y_i$ be the log returns and remove 688
observations where the return was $0$ (the price did not change from one day to
the next). Using the Kalman filter, the negative log likelihood is given by
\begin{equation}
\mathcal{L}(\Yin | \varsigma, \phi, \sigma_\eta^2) \propto
\sum_{i=1}^n \log F_i + v_i^2 F_i^{-1}.
\end{equation}
Minimizing this gives estimates $\varsigma=-9.62$, $\phi=0.996$, and $\sigma^2_\eta =
0.003$. Using absolute error loss gives training error $\hat{R}_n(f)=1.44$.

To actually calculate the bound, we need a few more values. First, using the
methods in \citet{estimating-beta-mixing,McDonaldShalizi2015}, we can estimate $\beta_8=0.017$. For
$a>8$, the optimal point estimate of $\beta_a$ is $0$.  While this is
presumably an underestimate, we will take $\beta_a=0$ for $a>8$.  For the upper
bound in \autoref{ass:B}, we use $M=\sqrt{\pi^2/2}$ which corresponds to the SV
model being true.

Combining these values with the VC dimension for the stochastic volatility
model, we can bound the prediction risk.  Finally, we take $\mu=538$, $a=11$,
$d=2$, and $\E[|Y_1|]=M$. The
result is the bound
$R_n(f) \leq 15.92$
with probability at least 0.95. In other words, the bound is much larger than
the training error, but this is to be expected: the data are highly dependent,
so the large $n$ translates into a relatively small effective sample size
$\mu$.

For comparison, we also computed the bound for forecasts produced with an AR(2)
model (with intercept) and with the global mean alone.  In the case of the
mean, we take $\mu=658$ and $a=9$ since in this case, $d=0$.  The results are
shown in \autoref{tab:one}. The stochastic volatility model reduces the
training error by 5\% relative to predicting with the mean, an improvement
which is marginal at best. But the resulting risk bound clearly demonstrates
that given the small effective sample size, even this gain may be spurious: it
is likely that the stochastic volatility model is simply over-fitting.

\begin{table}[t!]
 \centering
  \begin{tabular}{lcccc}
    \hline
    Model & Training error & AIC-Baseline &Risk bound
    ($1-\eta>0.95$)&Effective VC \\
    \hline
    SV & 1.44     & -2816 & 15.92 & 3\\
    AR(2) & 1.49 & -348 & 15.00 & 3\\
    Mean & 1.51 & 0 & 12.63 & 1\\
    \hline
  \end{tabular}
  \caption{This table shows the training error and risk bounds for 3
    models. AIC is
    given as the difference from predicting with the global mean (the smaller the value, the more
    support for that model). The ``Effective VC'' dimension column
    reports the exponent in the bound on the growth function. This is
    slightly different from the VC dimension of the model.}
  \label{tab:one}
\end{table}

\subsection{Real business cycle model}
\label{sec:real-business-cycle}

In this section, we will discuss the methodology for applying risk bounds to
the forecasts generated by the real business cycle (RBC) model. This is a
standard tool in macroeconomic forecasting.  For a discussion of the RBC model
and the standard methods used to bring such models to the data, see, for
example \citet{Romer-advanced-macro, DeJong-et-al-bayesian-approach,
  DeJong-Dave-structural-macro, Fernandez-Villaverde-on-DSGE-econometrics,
  KydlandPrescott1982, SmetsWouters2007, Sims-solving-linear-rat-exp}.

To estimate the parameters of this model, we use four data series. These are
GDP $y_t$, consumption $c_t$, investment $i_t$, and hours worked $n_t$
which are from the Federal Reserve Economic Database. The series we use are
shown in \autoref{fig:dsgedata}.

\begin{figure}
  \centering
  \includegraphics[width=4.5in]{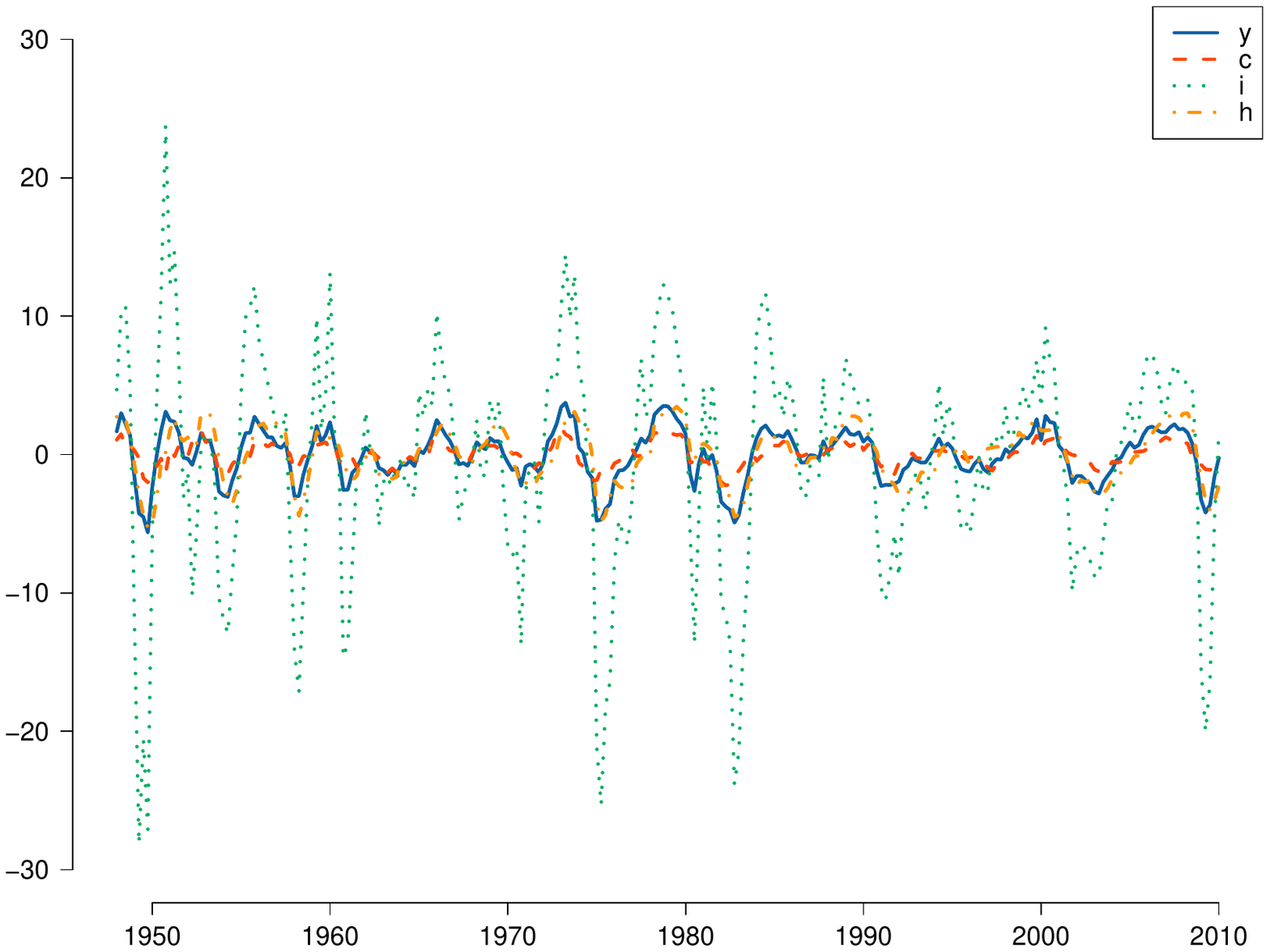}
  \caption{Time series used to estimate the RBC model. These are quarterly data
    from 1948:I until 2010:I. The blue line is GDP (output), the red line is
    consumption, the green line is investment, and the orange line is hours
    worked. These data are plotted as percentage deviations from trend as
    discussed in the Online Appendix.}
  \label{fig:dsgedata}
\end{figure}

The basic idea of the estimation is to transform the model from an
inter-temporal optimization form into a state space model.  This leads to a
linear, Gaussian state-space model with four observed variables (listed above),
and two unobserved state variables.  The mapping from parameters of the
optimization problem to parameters of the state-space model is nonlinear, but,
for each parameter setting, the Kalman filter returns the likelihood, so that
likelihood methods are possible.  As the data are fairly uninformative about
many of the parameters, we estimate by maximizing a penalized likelihood,
rather than a simple likelihood.  Then the Kalman filter produces in-sample
forecasts which are linear in past values of the data, so that we could
potentially apply the growing memory bound.

For macroeconomic time series, there is not enough data to give nontrivial
bounds, regardless of the mixing coefficients or the size of the finite memory
approximation.  \autoref{fig:dsgedata} shows $n=245$ observations.  The minimal
possible finite approximation model is a VAR with one lag and four
time series.  In this case, since we are
dealing with vector valued forecasts, we take $\ell(y-y')=\|y-y'\|_2$
so that the induced matrix norm of $A$ is $\ell^*(A) =
\sigma_{\max}(A)$, the largest singular value of $A$. We
assume that \autoref{ass:C} is satisfied with $M=0.1$ and demand confidence
$0.95$ ($\eta=0.05$). By \autoref{thm:vcd-ar}, this class has growth function bounded by
$(en)^8$.

Again, using the methods of \citet{estimating-beta-mixing}, we can estimate the
$\beta$-mixing coefficients of the macroeconomic data set.  The result is a
point estimate $\beta_4=0$.  Assuming that this is approximately accurate ($0$
is of course an underestimate), this suggests that the effective size of the
macroeconomic data set is no more than about $\mu=25$, much smaller then
$n=245$. To calculate the bound, we assume that $\E[\|Y\|_2] < 0.1$. The
training error of the fitted RBC model is $\hat{R}_n(f) = 0.046$. The
bound is $R_n(f) < 2.29$.

The bound here is three orders of magnitude larger than the training error. If
the bound is tight, then this suggests that the training error severely
underestimates the true prediction risk. Of course, this should not be too
surprising since the RBC model has 11 parameters and we are trying to get
confidence intervals using only 25 effective data points.

In some sense, the empirical results in this section may seem slightly
unreasonable.  Since the results are only upper bounds, it is important to get
an idea as to how tight they may be. We address this issue in the next section.

\section{Properties of our results}
\label{sec:how-loose-are}

In the previous section, we showed that the upper bound for the risk of
standard macroeconomic forecasting models may be large. This of course
raises the question ``How tight are these bounds?'' We address this question
next and then discuss how to use the bounds for model selection.

\subsection{How tight are the bounds?}
\label{sec:how-tight-are}

Here we give some idea of how tight the bounds presented in
\autoref{sec:risk-bounds} are.  Denote the function that minimizes the training
error (or penalized training error) over $\F$ by $\widehat{f}_{erm}$, and let
\begin{equation}
  f^* = \argmin_{f \in\F} R_n(f)
\end{equation}
be the oracle predictor with respect to $\F$.  We call
\begin{equation}
  \label{eq:oracle}
  L_n(\Pi) := \sup_{\P \in \Pi} \E_\P[R_n(\widehat{f}_{erm}) -
  R_n(f^*)] = \sup_{\P \in \Pi} \E_\P[R_n(\widehat{f}_{erm})] -
  R_n(f^*)
\end{equation}
the ``oracle loss''; it describes how well empirical risk minimization works
relative to the best possible predictor $f^*$ over the worst distribution $\P$.
\citet{Vapnik-theory} shows that for classification and i.i.d.\ data, for
sufficiently large $n$, there exist constants $\Upsilon_*$ and $\Upsilon^*$ such that
\begin{equation}
  \label{eq:rate}
  \Upsilon_*\sqrt{\frac{\vcd(\CF)}{n}} \leq L_n(\Pi) \leq
  \Upsilon^*\sqrt{  \frac{\vcd(\CF)\log n}{n} },
\end{equation}
where $\Pi$ is the class of all distributions on
$\mathbb{U}\times\{0,1\}$ and $\CF$ has finite VC-dimension $\vcd(\CF)$.
In other words, for i.i.d.\ data, the best we can hope to do is a rate of
$O\left(\sqrt{\frac{\vcd(\CF)}{n}}\right)$ and prediction methods which perform
worse than $O\left(\sqrt{\frac{\vcd(\CF) \log n}{n}}\right)$ are
inefficient. We will derive similar bounds for the $\beta$-mixing setting.
First, we need a slightly different version of \autoref{thm:bound1}.


\begin{theorem}
  \label{thm:boundedbeta}
  Suppose that $\ell(y-y')<M$, that \autoref{ass:A} holds, and that
  $\F$ has a fixed memory length $d<n$. Let $\mu$ and $a$ be integers such that
  $2\mu a+d\leq n$. Then, for all $\epsilon>0$,
  \begin{align}
    \label{eq.future:bound}
    \P\left( \sup_{f \in \F} (R_n(f) - \widehat{R}_n(f))
        > \epsilon \right)&\leq 8 \GF(n,\CloF)\exp\left\{-
        \frac{\mu\epsilon^2}{\Upsilon} \right\} + 2\mu\beta_{a-d}.
    \end{align}
    where $\Upsilon$ depends only on $M$.
\end{theorem}
The proof of \autoref{thm:boundedbeta} is exactly like that for
\autoref{thm:bound1} using the result for bounded loss in
\autoref{thm:vapnik} rather than that for relative loss in
\autoref{cor:corteseq} to control the probability that empirical risk
over blocks deviates from its expectation (see (\ref{eq:almost-finished-bound})).
\begin{assumption}
  \label{ass:expmix}
  The time series $\Yinf$ is \emph{geometrically} $\beta$-mixing, that is
  $ \beta_a = \upsilon_1 \exp(-\upsilon_2a^\kappa)$
  for some constants $\upsilon_1, \upsilon_2, \kappa$.
\end{assumption}

\begin{theorem}
  \label{thm:mixingoracle}
  Suppose $\ell(y-y')<M$ and that \autoref{ass:expmix} holds. Further
  assume that $\CloF$ has finite VC-dimension $\vcd(\CloF)$. Then, for
  sufficiently large $n$, there exist constants $\Upsilon_*$ and $\Upsilon^*$, independent of $n$
  and $\vcd(\CloF)$, such that
  \begin{equation}
    \Upsilon_*\sqrt{\frac{\vcd(\CloF)}{n}} \leq L_n(\Pi) \leq \Upsilon^*\sqrt{
        \frac{\vcd(\CloF) \log n}{ n^{\kappa/(1+\kappa)}}}.
  \end{equation}
\end{theorem}

If we instead assume \emph{algebraic mixing}, that is $\beta_a =
\upsilon a^{-\rho}$,
then we can retrieve the same rate where $0<\kappa<(\rho-1)/2$
\citep[see][]{Meir-nonparametric-time-series}. \autoref{thm:mixingoracle} says that in dependent data
settings, using the blocking approach developed here, we may pay a penalty: the
upper bound on $L_n(\Pi)$ goes to zero more slowly than in the i.i.d.\ case. But,
the lower bound cannot be made any tighter since i.i.d.\ processes are still
allowed under \autoref{ass:expmix} (and of course under the more general
\autoref{ass:A}). In other words, we may have $\kappa\rightarrow\infty$ so we
can not rule out the faster learning rate of $O\left(\sqrt{\frac{\vcd(\CloF)\log
      n}{n}}\right)$.

\subsection{Structural risk minimization}
\label{sec:struct-risk-minim}

Our presentation so far has focused on choosing one function $\hat{f}$ from a
model $\F$ and demonstrating that the prediction risk $R_n(\hat{f})$ is well
characterized by the training error inflated by a complexity term. The
procedure for actually choosing $\hat{f}$ has been ignored. Common ways of
choosing $\hat{f}$ are frequently referred to as \emph{empirical risk
  minimization} or ERM: approximate the expected risk $R_n(f)$ with the
empirical risk $\hat{R}_n(f)$, and choose $\hat{f}$ to minimize the empirical
risk. Many likelihood based methods have exactly this flavor. But more
frequently, forecasters have many different models in mind, each with a
different empirical risk minimizer. Regularized model classes (ridge
regression, lasso, Bayesian methods) implicitly have this structure ---
altering the amount of regularization leads to different models $\F$. Or one
may have many different forecasting models from which the forecaster would like
to choose the best. This scenario leads to a generalization of ERM called
\emph{structural risk minimization} or SRM.

Given a collection of models $\F_1, \F_2,\ldots$ each with associated empirical
risk minimizers $\hat{f}_1, \hat{f}_2,\ldots$, we wish to use the function
which has the smallest risk. Of course different models have different
complexities, and those with larger complexities will tend to have smaller
empirical risk. To choose the best function, we therefore penalize the
empirical risk and select that function which minimizes the penalized
version. Model selection tools like AIC or BIC have exactly this form, but they
rely on specific knowledge of the data likelihood and use asymptotics to derive
approximate penalties.  In contrast, we have finite-sample bounds for the
expected risk.  This leads to a natural model selection rule: choose the
predictor which has the smallest bound on the expected risk.

The generalization error bounds in \autoref{sec:risk-bounds} allow
one to perform model selection via the SRM principle without knowledge of the
likelihood or appeals to asymptotic results. The penalty accounts for the
complexity of the model through the VC dimension. Most useful however is that
by using generalization error bounds for model selection, we are minimizing the
prediction risk.
So in the volatility forecasting exercise above, we would
choose the mean.

If we want to make the prediction risk as small as possible, we can minimize
the generalization error bound simultaneously over models $\F$ and functions
within those models. This amounts to treating VC dimension as a control
variable. Therefore, by minimizing both the empirical risk and the VC
dimension, we can choose that model and function which has the smallest
prediction risk, a claim which other model selection procedures cannot make
\citep{Vapnik-nature, Massart-concentration-and-selection}.

\section{Conclusion}
\label{sec:conclusion}

This paper demonstrates how to control the generalization error of common
time-series forecasting models, especially those used in economics and
engineering---ARMA models, vector autoregressions (Bayesian or otherwise),
linearized dynamic stochastic general equilibrium models, and linear
state-space models.  We derive upper bounds on the risk, which hold with high
probability while requiring only weak assumptions on the data-generating
process.  These bounds are finite sample in nature, unlike standard model
selection penalties such as AIC or BIC.  Furthermore, they do not suffer the
biases inherent in other risk estimation techniques such as the pseudo-cross
validation approach often used in the economic forecasting literature.

While we have stated these results in terms of standard economic forecasting
models, they have very wide applicability. \autoref{thm:bound1} applies to any
forecasting procedure with fixed memory length, linear or non-linear.
\autoref{thm:bound3} applies only to methods whose forecasts are linear in the
observations, but a similar result for nonlinear methods would just need to
ensure that the dependence of the forecast on the past decays in some suitable
way.

Rather than deriving bounds theoretically, one could attempt to estimate bounds
on the risk.  While cross-validation is tricky
\citep{Racine-consistent-cv-for-dependent-data}, nonparametric bootstrap
procedures may do better.  A fully nonparametric version is possible, using the
circular bootstrap \citep[reviewed in][]{Lahiri-block-bootstraps-compared}.
Bootstrapping lengthy out-of-sample sequences for testing fitted model
predictions yields intuitively sensible estimates of $R_n(f)$, but there is
currently no theory about the coverage level.  Also, while models like VARs can
be fit quickly to simulated data, general state-space models, let alone DSGEs,
require large amounts of computational power, which is an obstacle to any
resampling method.

While our results are a crucial first step for the learning-theoretic analysis
of time-series forecasts, many avenues remain for future exploration.  To gain
a more complete picture of the performance of forecasting algorithms, we would
want minimax lower bounds \citep[c.f.][]{Tsybakov-intro}.  These would tell us
the smallest risk we could hope to achieve using any forecaster in some larger
model class, letting us ask whether any of the models in common use actually
approach this minimum. Another possibility is to target not the {\em ex ante}
risk of the forecast, but the {\em ex post} regret: how much better might our
forecasts have been, in retrospect and on the actually-realized data, had we
used a different prediction function from the model $\F$
\citep{prediction-learning-and-games,Rakhlin-Sridharan-Tewari-online-learning}?
Remarkably, we can find forecasters which have low {\em ex post} regret, even
if the data came from an adversary trying to make us perform badly.  If we
target regret rather than risk, we can actually ignore mixing, and even
stationarity \citep{Growing-ensembles}.

An increased recognition of the abilities and benefits of statistical learning
theory can be of tremendous aid to financial and economic forecasters. The
results presented here represent an initial yet productive foray in this
direction. They allow for principled model comparisons as well as high
probability performance guarantees.  Future work in this direction will only
serve to sharpen our ability to measure predictive power.




\bibliography{RiskBounds}

\begin{thebibliography}{57}
\providecommand{\natexlab}[1]{#1}
\providecommand{\url}[1]{\texttt{#1}}
\expandafter\ifx\csname urlstyle\endcsname\relax
  \providecommand{\doi}[1]{doi: #1}\else
  \providecommand{\doi}{doi: \begingroup \urlstyle{rm}\Url}\fi

\bibitem[Adams and Nobel(2010)]{Adams-Nobel-VC-classes-under-ergodic}
Terrence~M. Adams and Andrew~B. Nobel.
\newblock Uniform convergence of {Vapnik}-{Chervonenkis} classes under ergodic
  sampling.
\newblock \emph{Annals of Probability}, 38:\penalty0 1345--1367, 2010.

\bibitem[Agarwal and Duchi(2013)]{AgarwalDuchi2013}
Alekh Agarwal and John~C. Duchi.
\newblock The generalization ability of online algorithms for dependent data.
\newblock \emph{{IEEE} Transactions on Information Theory}, 59\penalty0
  (1):\penalty0 573--587, 2013.

\bibitem[Alquier et~al.(2012)Alquier, Wintenberger,
  et~al.]{AlquierWintenberger2012}
Pierre Alquier, Olivier Wintenberger, et~al.
\newblock Model selection for weakly dependent time series forecasting.
\newblock \emph{Bernoulli}, 18\penalty0 (3):\penalty0 883--913, 2012.

\bibitem[Alquier et~al.(2014)Alquier, Li, and Wintenberger]{AlquierLi2014}
Pierre Alquier, Xiaoyin Li, and Olivier Wintenberger.
\newblock Prediction of time series by statistical learning: general losses and
  fast rates.
\newblock \emph{Dependence Modelling}, 1:\penalty0 65--93, 2014.

\bibitem[Anderson and Moore(2012)]{AndersonMoore2012}
Brian~D.O. Anderson and John~B. Moore.
\newblock \emph{Optimal filtering}.
\newblock Prentice-Hall, Englewood Cliffs, NJ, 2012.

\bibitem[Anthony and Bartlett(1999)]{Anthony-Bartlett-neural-network-learning}
Martin Anthony and Peter~L. Bartlett.
\newblock \emph{Neural Network Learning: Theoretical Foundations}.
\newblock Cambridge University Press, Cambridge, England, 1999.

\bibitem[Bartlett and
  Mendelson(2002)]{Bartlett-Mendelson-on-Rademacher-complexity}
Peter~L. Bartlett and Shahar Mendelson.
\newblock Rademacher and {Gaussian} complexities: Risk bounds and structural
  results.
\newblock \emph{Journal of Machine Learning Research}, 3:\penalty0 463--482,
  2002.

\bibitem[Bousquet and Elisseeff(2002)]{Bousquet-Elisseeff-stability}
Olivier Bousquet and Andr{\'e} Elisseeff.
\newblock Stability and generalization.
\newblock \emph{Journal of Machine Learning Research}, 2:\penalty0 499--526,
  2002.

\bibitem[Bradley(2005)]{Bradley-strong-mixing}
Richard~C. Bradley.
\newblock Basic properties of strong mixing conditions. {A} survey and some
  open questions.
\newblock \emph{Probability Surveys}, 2:\penalty0 107--144, 2005.

\bibitem[Carrasco and
  Chen(2002)]{Carrasco-and-Chen-mixing-of-GARCH-and-stoch-vol}
Marine Carrasco and Xiaohong Chen.
\newblock Mixing and moment properties of various {GARCH} and stochastic
  volatility models.
\newblock \emph{Econometric Theory}, 18:\penalty0 17--39, 2002.

\bibitem[Cesa-Bianchi and Lugosi(2006)]{prediction-learning-and-games}
Nicol{\`o} Cesa-Bianchi and G{\'a}bor Lugosi.
\newblock \emph{Prediction, Learning, and Games}.
\newblock Cambridge University Press, Cambridge, England, 2006.

\bibitem[Corless et~al.(1996)Corless, Gonnet, Hare, Jeffrey, and
  Knuth]{Corless-et-al-on-Lambert-W}
Robert~M. Corless, Gaston~H. Gonnet, D.~E.~G. Hare, David~J. Jeffrey, and
  Donald~E. Knuth.
\newblock On the {L}ambert $w$ function.
\newblock \emph{Advances in Computational Mathematics}, 5:\penalty0 329--359,
  1996.

\bibitem[Cortes et~al.(2010)Cortes, Mansour, and
  Mohri]{Cortes-Mansour-Mohri-bounds}
Corinna Cortes, Yishay Mansour, and Mehryar Mohri.
\newblock Learning bounds for importance weights.
\newblock In  \citet{nips-2010}, pages 442--450.

\bibitem[DeJong and Dave(2007)]{DeJong-Dave-structural-macro}
David~N. DeJong and Chetan Dave.
\newblock \emph{Structural Macroeconometrics}.
\newblock Princeton University Press, Princeton, New Jersey, 2007.

\bibitem[DeJong et~al.(2000)DeJong, Ingram, and
  Whiteman]{DeJong-et-al-bayesian-approach}
David~N. DeJong, Beth~F. Ingram, and Charles~H. Whiteman.
\newblock A {B}ayesian approach to dynamic macroeconomics.
\newblock \emph{Journal of Econometrics}, 98:\penalty0 203--223, 2000.

\bibitem[Doukhan(1995)]{Doukhan-on-mixing}
Paul Doukhan.
\newblock \emph{Mixing: Properties and Examples}.
\newblock Springer-Verlag, New York, 1995.

\bibitem[Durbin and Koopman(2001)]{Durbin-Koopman-state-space-methods}
James Durbin and Siem~Jam Koopman.
\newblock \emph{Time Series Analysis by State Space Methods}.
\newblock Oxford University Press, Oxford, 2001.

\bibitem[Engle(1982)]{Engle-on-ARCH}
Robert~F. Engle.
\newblock Autoregressive conditional heteroscedasticity with estimates of
  variance of {United} {Kingdom} inflation.
\newblock \emph{Econometrica}, 50:\penalty0 987--1008, 1982.

\bibitem[Engle(2001)]{Engle-GARCH-101}
Robert~F. Engle.
\newblock Garch 101: The use of arch/garch models in applied econometrics.
\newblock \emph{Journal of Economic Perspectives}, 15:\penalty0 157--168, 2001.

\bibitem[Fern{\'a}ndez-Villaverde(2010)]{Fernandez-Villaverde-on-DSGE-econometrics}
Jes\'us Fern{\'a}ndez-Villaverde.
\newblock The econometrics of {DSGE} models.
\newblock \emph{{SERIEs}}, 1:\penalty0 3--49, 2010.

\bibitem[Harvey et~al.(1994)Harvey, Ruiz, and
  Shephard]{Harvey-Ruiz-multivariate-stoch-var}
Andrew Harvey, Esther Ruiz, and Neil Shephard.
\newblock Multivariate stochastic variance models.
\newblock \emph{The Review of Economic Studies}, 61:\penalty0 247--264, 1994.

\bibitem[Kalman(1960)]{Kalman}
Rudolf~E. Kalman.
\newblock A new approach to linear filtering and prediction problems.
\newblock \emph{{ASME} Transactions, Journal of Basic Engineering},
  82D:\penalty0 35--50, 1960.

\bibitem[Karandikar and Vidyasagar(2002)]{Karandikar-Vidyasagar-rates-of-UCEM}
Rajeeva~L. Karandikar and Mathukumalli Vidyasagar.
\newblock Rates of uniform convergence of empirical means with mixing
  processes.
\newblock \emph{Statistics and Probability Letters}, 58:\penalty0 297--307,
  2002.

\bibitem[Kearns and Ron(1999)]{Kearns-Ron-algorithmic-stability}
Michael~J. Kearns and Dana Ron.
\newblock Algorithmic stability and sanity-check bounds for leave-one-out
  cross-validation.
\newblock \emph{Neural Computation}, 11:\penalty0 1427--1453, 1999.

\bibitem[Kleijn and van~der Vaart(2006)]{Kleijn-van-der-Vaart}
B.~J.~K. Kleijn and Aad~W. van~der Vaart.
\newblock Misspecification in infinite-dimensional {Bayesian} statistics.
\newblock \emph{Annals of Statistics}, 34:\penalty0 837--877, 2006.

\bibitem[Koltchinskii and Panchenko(2002)]{koltchinskii2002}
V.~Koltchinskii and D.~Panchenko.
\newblock Empirical margin distributions and bounding the generalization error
  of combined classifiers.
\newblock \emph{Ann. Statist.}, 30\penalty0 (1):\penalty0 1--50, 02 2002.

\bibitem[Kuznetsov and Mohri(2014)]{KuznetsovMohri2014}
Vitaly Kuznetsov and Mehryar Mohri.
\newblock Generalization bounds for time series prediction with non-stationary
  processes.
\newblock In Peter Auer, Alexander Clark, Thomas Zeugmann, and Sandra Zilles,
  editors, \emph{Algorithmic Learning Theory}, volume 8776 of \emph{Lecture
  Notes in Computer Science}, pages 260--274. Springer International
  Publishing, 2014.

\bibitem[Kydland and Prescott(1982)]{KydlandPrescott1982}
Finn~E. Kydland and Edward~C. Prescott.
\newblock Time to build and aggregate fluctuations.
\newblock \emph{Econometrica}, 50\penalty0 (6):\penalty0 1345--1370, November
  1982.

\bibitem[Lafferty et~al.(2010)Lafferty, Williams, Shawe-Taylor, Zemel, and
  Culotta]{nips-2010}
John Lafferty, C.~K.~I. Williams, John Shawe-Taylor, Richard~S. Zemel, and
  A.~Culotta, editors.
\newblock \emph{Advances in Neural Information Processing 23 [NIPS 2010]},
  Cambridge, Massachusetts, 2010. MIT Press.

\bibitem[Lahiri(1999)]{Lahiri-block-bootstraps-compared}
Soumendra~N. Lahiri.
\newblock Theoretical comparisons of block bootstrap methods.
\newblock \emph{Annals of Statistics}, 27:\penalty0 386--404, 1999.

\bibitem[Massart(2007)]{Massart-concentration-and-selection}
Pascal Massart.
\newblock \emph{Concentration Inequalities and Model Selection}.
\newblock Springer-Verlag, Berlin, 2007.

\bibitem[McDonald et~al.(2011{\natexlab{a}})McDonald, Shalizi, and
  Schervish]{estimating-beta-mixing}
Daniel~J. McDonald, Cosma~Rohilla Shalizi, and Mark Schervish.
\newblock Estimating $\beta$-mixing coefficients.
\newblock In Geoffrey Gordon, David Dunson, and Miroslav Dud{\'\i}k, editors,
  \emph{Proceedings of the Fourteenth International Conference on Artificial
  Intelligence and Statistics}, volume~15. JMLR W\&CP, 03 2011{\natexlab{a}}.
\newblock URL \url{http://arxiv.org/abs/1103.0941}.

\bibitem[McDonald et~al.(2011{\natexlab{b}})McDonald, Shalizi, and
  Schervish]{risk-bounds-for-ar-models}
Daniel~J. McDonald, Cosma~Rohilla Shalizi, and Mark Schervish.
\newblock Risk bounds for autoregressive models.
\newblock Technical report, Statistics Department, CMU, 2011{\natexlab{b}}.

\bibitem[McDonald et~al.(2015)McDonald, Shalizi, and
  Schervish]{McDonaldShalizi2015}
Daniel~J. McDonald, Cosma~Rohilla Shalizi, and Mark Schervish.
\newblock Estimating beta-mixing coefficients via histograms.
\newblock \emph{Electronic Journal of Statistics}, 9:\penalty0 2855--2883,
  2015.

\bibitem[Meir(2000)]{Meir-nonparametric-time-series}
Ron Meir.
\newblock Nonparametric time series prediction through adaptive model
  selection.
\newblock \emph{Machine Learning}, 39:\penalty0 5--34, 2000.

\bibitem[Modha and Masry(1998)]{ModhaMasry1998}
D.S. Modha and E.~Masry.
\newblock Memory-universal prediction of stationary random processes.
\newblock \emph{IEEE Transactions on Information Theory}, 44\penalty0
  (1):\penalty0 117--133, 1998.

\bibitem[Mohri and
  Rostamizadeh(2009)]{Mohri-Rostamizadeh-rademacher-for-non-iid}
Mehryar Mohri and Afshin Rostamizadeh.
\newblock Rademacher complexity bounds for non-{I.I.D.} processes.
\newblock In Daphne Koller, D.~Schuurmans, Y.~Bengio, and L{\'e}on Bottou,
  editors, \emph{Advances in Neural Information Processing Systems 21 [NIPS
  2008]}, pages 1097--1104, 2009.

\bibitem[Mohri and Rostamizadeh(2010)]{Mohri-Rostamizdaeh-stability-bounds}
Mehryar Mohri and Afshin Rostamizadeh.
\newblock Stability bounds for stationary $\phi$-mixing and $\beta$-mixing
  processes.
\newblock \emph{Journal of Machine Learning Research}, 11, 2010.

\bibitem[Mokkadem(1988)]{Mokkadem1988}
A.~Mokkadem.
\newblock Mixing properties of {ARMA} processes.
\newblock \emph{Stochastic Processes and their Applications}, 29\penalty0
  (2):\penalty0 309--315, 1988.

\bibitem[M{\"u}ller(2013)]{Muller-risk-of-Bayesian-inference}
Ulrich~K. M{\"u}ller.
\newblock Risk of {Bayesian} inference in misspecified models, and the sandwich
  covariance matrix.
\newblock \emph{Econometrica}, 81:\penalty0 1805--1849, 2013.

\bibitem[Nobel and Dembo(1993)]{Nobel-Dembo-uniform-laws-of-averages}
Andrew Nobel and Amir Dembo.
\newblock A note on uniform laws of averages for dependent processes.
\newblock \emph{Statistics and Probability Letters}, 17:\penalty0 169--172,
  1993.

\bibitem[Pollard(1984)]{Pollard-convergence}
David Pollard.
\newblock \emph{Convergence of Stochastic Processes}.
\newblock Springer-Verlag, Berlin, 1984.

\bibitem[Pollard(1990)]{Pollard-empirical-processes}
David Pollard.
\newblock \emph{Empirical Processes: Theory and Applications}, volume~2 of
  \emph{NSF-CBMS Regional Conference Series in Probability and Statistics}.
\newblock Institute of Mathematical Statistics, Hayward, California, 1990.

\bibitem[Racine(2000)]{Racine-consistent-cv-for-dependent-data}
Jeff Racine.
\newblock Consistent cross-validatory model-selection for dependent data:
  hv-block cross-validation.
\newblock \emph{Journal of Econometrics}, 99:\penalty0 39--61, 2000.

\bibitem[Rakhlin et~al.(2010)Rakhlin, Sridharan, and
  Tewari]{Rakhlin-Sridharan-Tewari-online-learning}
Alexander Rakhlin, Karthik Sridharan, and Ambuj Tewari.
\newblock Online learning: Random averages, combinatorial parameters, and
  learnability.
\newblock In  \citet{nips-2010}, pages 1984--1992.

\bibitem[Romer(2011)]{Romer-advanced-macro}
David Romer.
\newblock \emph{Advanced Macroeconomics}.
\newblock McGraw-Hill, New York, 4 edition, 2011.

\bibitem[Shalizi(2009)]{CRS-dynamics-of-bayes}
Cosma~Rohilla Shalizi.
\newblock Dynamics of {Bayesian} updating with dependent data and misspecified
  models.
\newblock \emph{Electronic Journal of Statistics}, 3:\penalty0 1039--1074,
  2009.

\bibitem[Shalizi et~al.(2011)Shalizi, Jacobs, Klinkner, and
  Clauset]{Growing-ensembles}
Cosma~Rohilla Shalizi, Abigail~Z. Jacobs, Kristina~Lisa Klinkner, and Aaron
  Clauset.
\newblock Adapting to non-stationarity with growing expert ensembles.
\newblock Technical report, Statistics Department, CMU, 2011.

\bibitem[Sims(2002)]{Sims-solving-linear-rat-exp}
Christopher~A. Sims.
\newblock Solving linear rational expectations models.
\newblock \emph{Computational Economics}, 20:\penalty0 1--20, 2002.

\bibitem[Smets and Wouters(2007)]{SmetsWouters2007}
Frank Smets and Rafael Wouters.
\newblock Shocks and frictions in {US} business cycles: {A Bayesian DSGE}
  approach.
\newblock \emph{American Economic Review}, 97\penalty0 (3):\penalty0 586--606,
  June 2007.

\bibitem[Steinwart and Christmann(2009)]{SteinwartChristmann2009}
Ingo Steinwart and Andreas Christmann.
\newblock Fast learning from non-i.i.d. observations.
\newblock In Y.~Bengio, D.~Schuurmans, J.~Lafferty, C.~K.~I. Williams, and
  A.~Culotta, editors, \emph{Advances in Neural Information Processing
  Systems}, volume~22, pages 1768--1776. MIT Press, Cambridge, MA, 2009.

\bibitem[Tsybakov(2008)]{Tsybakov-intro}
Alexandre~B. Tsybakov.
\newblock \emph{Introduction to Nonparametric Estimation}.
\newblock Springer Verlag, New York, 2008.

\bibitem[Vapnik(1998)]{Vapnik-theory}
Vladimir~N. Vapnik.
\newblock \emph{Statistical Learning Theory}.
\newblock Wiley, New York, 1998.

\bibitem[Vapnik(2000)]{Vapnik-nature}
Vladimir~N. Vapnik.
\newblock \emph{The Nature of Statistical Learning Theory}.
\newblock Springer-Verlag, Berlin, 2nd edition, 2000.

\bibitem[Vapnik and Chervonenkis(1971)]{Vapnik-Chervonenkis-1971}
Vladimir~N. Vapnik and Alexey~Y. Chervonenkis.
\newblock On the uniform convergence of relative frequencies of events to their
  probabilities.
\newblock \emph{Theory of Probability and its Applications}, 16:\penalty0
  264--280, 1971.

\bibitem[Vidyasagar(2003)]{Vidyasagar-on-learning-and-generalization}
Mathukumalli Vidyasagar.
\newblock \emph{Learning and Generalization: With Applications to Neural
  Networks}.
\newblock Springer-Verlag, Berlin, second edition, 2003.

\bibitem[Yu(1994)]{Yu-rates-of-convergence}
Bin Yu.
\newblock Rates of convergence for empirical processes of stationary mixing
  sequences.
\newblock \emph{Annals of Probability}, 22:\penalty0 94--116, 1994.

\end{thebibliography}
\clearpage
\appendix

\section{Auxiliary results}
\label{sec:lemmas}

\begin{lemma}[\citealp{Yu-rates-of-convergence} Lemma 4.1]
  \label{lem:yu}
  Let $Z$ be an event with respect to the block sequence
  $\mathbf{U}$. Then,
  \begin{equation}
    |\P(Z) - \widetilde{\P}(Z)| \leq \beta_a(\mu-1),
  \end{equation}
  where the first probability is with respect to the dependent block
  sequence, $\mathbf{U}$, and
  $\widetilde{\P}$ is with respect to the independent sequence, $\mathbf{U}'$.
\end{lemma}
This lemma essentially gives a way to apply i.i.d.\ results to $\beta$-mixing
data. Because the dependence decays as we increase the separation between
blocks, widely spaced blocks are nearly independent of each other. In
particular, the difference between expectations over these nearly independent
blocks and expectations over blocks which are actually independent can be
controlled by the $\beta$-mixing coefficient.  Very similar results are also
given in \citet{Nobel-Dembo-uniform-laws-of-averages} and
\citet{Vidyasagar-on-learning-and-generalization}.

\begin{lemma}
  [\citealp{Cortes-Mansour-Mohri-bounds} Theorem 7]
  \label{lem:vapnik}
  Under \autoref{ass:B},
  \begin{align}
    \P\left( \sup_{f\in\F} \frac{R_n(f) - \widehat{R}_n(f)}
      {Q_d(f)} > \epsilon\sqrt{2 + \log \frac{1}{\epsilon}}
    \right) &\leq 4G(2n,\CloF) \exp\left\{
      -\frac{n\epsilon^2}{4}\right\}.
  \end{align}
\end{lemma}

\begin{corollary}
  \label{cor:corteseq}
  Under \autoref{ass:B}, for $0<\epsilon\leq e^{3/2}/\sqrt{2}$,
  \begin{align}
    \P\left( \sup_{f\in\F} \frac{R_n(f) - \widehat{R}_n(f)}
      {Q_d(f)} > \epsilon\right) &\leq 4 G(2n,\CloF)\exp\left\{ -
        \frac{n\exp\left(W\left(-\frac{2\epsilon^2}{e^4}\right)+4\right)}{4} \right\}.
  \end{align}
\end{corollary}
\begin{proof}
  Letting $\epsilon'=\epsilon\sqrt{2 + \log \frac{1}{\epsilon}}$ and
  solving for $\epsilon$ gives
  \begin{align}
  \epsilon=\exp\left\{\frac{1}{2}\left(W\left(-\frac{2\epsilon'^2}{e^4}\right)+4\right)\right\}
  \end{align}
  if $0<\epsilon'\leq e^{3/2}/\sqrt{2}$, so
  \begin{align}
    \P\left( \sup_{f\in\F} \frac{R_n(f) - \widehat{R}_n(f)}
      {Q_d(f)} > \epsilon\sqrt{2 + \log \frac{1}{\epsilon}}
    \right) &\leq 4G(2n,\CloF) \exp\left\{
      -\frac{n\epsilon^2}{4}\right\}\\
    \Rightarrow \P\left( \sup_{f\in\F} \frac{R_n(f) - \widehat{R}_n(f)}
      {Q_d(f)} > \epsilon
    \right) &\leq 4G(2n,\CloF) \exp\left\{ -
      \frac{n\exp\left(W\left(-\frac{2\epsilon^2}{e^4}\right)+4\right)}{4}
    \right\}.
  \end{align}
\end{proof}

\begin{lemma}
  \label{lem:vcequiv}
  Suppose we have a function class $\F$ with growth function $G(n,
  \F)$. Let $2\mu a = n$. Consider the class
  $\F^\Sigma = \{ \sum_{j=1}^a f(z_{j}) : f \in \F\}$.
  $$G(\mu, \F^\Sigma) \leq G(n/2, \F) .$$
\end{lemma}

\begin{proof} Let $z_1,\ldots,z_\mu \in \mathbb{U}^{a}$,
$b_1,\ldots,b_\mu\in \mathbb{R}$ be a
set such that $\exists f^\Sigma_1,\ldots,f^\Sigma_K \in \F^\Sigma$ with $K =
G(\mu,\F^\Sigma)$ which shatters$\{z_i\}$. That is, for $v_{ik} \in \{0,1\}$, $i=1,\ldots,\mu$,
$k=1,\ldots,K$, we have
\[
I(f^\Sigma_k(z_i) > b_i) \mbox{ iff } v_{ik} = 1.
\]
But this means, $I(a^{-1}\sum_{j=1}^a f_k(z_{ij}) >b_i/a)$ iff
$v_{ik}=1$. Since this is a convex combination of the $f_k(z_{ij})$,
there is a subset of the $z_{ij}$ such that
\[
I(f_k(z_{ij})>b_i/a) \mbox{ iff } v_{ik} = 1.
\]
Thus $G(\mu, \F^\Sigma) \leq G(\mu a, \F) = G(n/2, \F)$.
\end{proof}

\begin{lemma}
  \label{lem:monotone-vc}
  Let $\mathcal{F}$ be a class of predictor functions $\mathbb{U} \mapsto
  \mathbb{R}$, and let $h: \mathbb{R} \mapsto
  \mathbb{R}$, be a
  monotone-increasing function. Define $h \circ \F = \{h\circ f: f \in
  \F\}$. Then $\forall n\geq 1$, $\GF(n,h\circ \F)\leq \GF(n,\F)$. 
\end{lemma}
\begin{proof}
  Let $z_1,\ldots,z_n \in \mathbb{U}$,
  $b_1,\ldots,b_n\in \mathbb{R}$ be a
  set such that $\exists f_1,\ldots, f_K \in\F$ and $K =
  \GF(\mu,h\circ \F)$ which shatters $\{z_i\}$. That is, for $v_{ik} \in \{0,1\}$, $i=1,\ldots,n$,
  $k=1,\ldots,K$, we have
  \[
  I(h(f_k(z_i)) > b_i) \mbox{ iff } v_{ik} = 1.
  \]
  Set
  \[
  c_i = \max_{k \in 1:K} \left\{ f_k(z_i) ~:~ h (f_k(u_i)) \leq b_i
  \right\}
  \]
  Then, because $h$ is monotone,
  \[
  b_i \geq h(f_k(u_i)) ~ \Leftrightarrow ~ c_i \geq f_k(u_i)
  \]
  Hence this set of points is shattered by the same collection of functions, and
  $\GF(n,h\circ \F)\leq \GF(n,\F)$.

\end{proof}




\section{Proofs of selected results}
\label{sec:proofs-results-srefs}
\begin{proof}[\autoref{lem:kalmanExpDecay}]
  We have that 
  \begin{align}
    b_{ij} &=Z\prod_{k=j+1}^{i-1} L_kK_j\\
                     &= Z  (I-P_{j+1}Z^\top F_{j+1}Z)T
                       (I-P_{j+2}Z^\top F_{j+2}Z)\cdots T
                       (I-P_{i-1}Z'F_{i-1}Z) K_j.\\
  \end{align}
  Now, by assumption $\lambda_{max}(T)<1$. Furthermore,
  $\lambda_{max}\left((I-P_{k}Z'F_{k}Z)\right)\leq 1$ for all $k$. To see
  this, write
  \begin{align}
    I-P_kZ^\top F_kZ &= I-P_kZ^\top (ZP_kZ^\top +H)^{-1}Z\\
    &= I-Z^\dagger [ZP_kZ^\top][(ZP_kZ^\top+H)^{-1}]Z,
  \end{align}
  where $Z^\dagger$ is a generalized inverse for $Z$.
  As, $P_k$ and $H$ are positive definite, $[ZP_kZ'][(ZP_kZ'+H)^{-1}]$ is
  positive definite with
  $\lambda_{\max}([ZP_kZ^\top ][(ZP_kZ^\top+H)^{-1}])<1$ and therefore $I-Z^\dagger
  [ZP_kZ^\top][(ZP_kZ^\top+H)^{-1}]Z$ is positive semi-definite if
  $dim(H)>dim(P_k)$ and positive definite if $dim(P_k)\leq
  dim(H)$. Therefore, $0<\lambda_{1}\left((I-P_{k}Z'F_{k}Z)\right)\leq
  1$. Finally, $\lambda_{1}(b_{i,j})\leq
  ||Z||_2 ||T||_2^{i-j-1}\prod_{k=j+1}^{i-1} ||I-P_{k}Z'F_{k}Z||_2
  \leq \lambda_{\max}(T)^{i-j-1}||Z||_2=O(r^{i-j-1})$.
\end{proof}
  
\begin{proof}[\autoref{lem:kalmanConv}]
  \begin{enumerate}
  \item This result is given in \citet[][4.4]{AndersonMoore2012}.
  \item The convergence of $\hat{P}_i$ is in
    \citet[][4.4]{AndersonMoore2012}. As $K_i = T\hat{P}_iZ^\top
    (Z\hat{P}_iZ^\top+H)^{-1}$ is continuous in $\hat{P}_i$, it
    converges as well.
  \item This result is given in \citet[][4.4]{AndersonMoore2012}.
  \item $\lVert \hat{P}_i-\overline{P} \rVert= O(\lambda_{\max}(T-\overline{K}Z)^i)$ is given
    \citet[][4.4]{AndersonMoore2012}. 
    For $F_i$,
    \begin{align*}
      \lVert F_i - \overline{F}\rVert &= \lVert F_i(F^{-1}_i -
                                       \overline{F}^{-1})\overline{F}\rVert
                                       =O(\lVert
                                       F^{-1}_i-\overline{F}^{-1}\rVert)\\  
      \intertext{since $\lambda_{\min}(F^{-1}_i)>\lambda_{\min}(H)
      >0 \Rightarrow 
      \lambda_{\max}(F_i)< 1/\lambda_{\min}(H)=O(1)$. Proceeding,}
                           O(\lVert F^{-1}_i-F^{-1}\rVert)&= O(\lVert
                                                            Z\hat{P}_iZ^\top-Z\overline{P}Z^\top\rVert) 
                           =O(\lVert \hat{P}_i-\overline{P}\rVert)
                           =O(\rho^i).
    \end{align*}
    For $K_i$, note that
    \begin{align*}
      K_i-K &= TP_iZ^\top F_i-T\overline{P}Z^\top \overline{F} 
           = TP_iZ^\top F_i - T\overline{P}Z^\top F_i
              +T\overline{P}Z^\top F_i - T\overline{P}Z^\top
              \overline{F}. 
    \end{align*}
    We have $\lVert TP_iZ^\top F_i - T\overline{P}Z^\top F_i\rVert=O(\rho^i)$ by the result
    for $\hat{P}_i$ and $\lVert T\overline{P}Z^\top F_i - T\overline{P}Z^\top \overline{F}\rVert =
    O(\rho^i)$ by the result for $F_i$.
  \end{enumerate}
\end{proof}

\begin{proof}[\autoref{lem:filterWeightsConv}]
  By \autoref{eq:smatrix}, $s_{ij} = Z(T-\overline{K}Z)^{i-j-1}\overline{K}$ with
  $\lambda_{\max}(T-\overline{K}Z)=\rho<1$. Thus $\lVert s_{ij} \rVert
  = \lVert s_{i-j-1} \rVert = O(\rho^{i-j-1})$. For $j<i-1$,
  \begin{align*}
    b_{ij} &= Z\prod_{k=j+1}^{i-1} (T-K_kZ)K_{j} &
    s_{ij} &= Z(T-\overline{K}Z)^{i-j-1} \overline{K}.
  \end{align*}
  Therefore,
  \begin{align*}
    \lVert b_{ij} - s_{ij}\rVert 
                    &\leq \norm{Z}\norm{\prod_{k=j+1}^{i-1}
                      (T-K_kZ)K_{j} - (T-\overline{K}Z)^{i-j-1}\overline{K}}\\ 
                    &\leq \norm{Z}\norm{\prod_{k=i-j}^{i-1}
                      (T-K_kZ)K_{j} -\prod_{k=j+1}^{i-1}
                      (T-K_kZ)\overline{K}}\\&\quad + \norm{Z}\norm{\prod_{k=j+1}^{i-1}
                      (T-K_kZ)\overline{K} - (T-\overline{K}Z)^{i-j-1}\overline{K}}\\ 
                    &= O(\lambda_{\max}(T-\overline{K}Z)^{j}) + O(\norm{\prod_{k=j+1}^{i-1}
                      (T-K_kZ) - (T -\overline{K}Z)^{i-j-1}}\\
                    &= O(\lambda_{\max}(T-\overline{K}Z)^{j}) + O(\norm{
                      (T-K_jZ)^{i-j-1} - (T -\overline{K}Z)^{i-j-1}})\\
                    &=O(\lambda_{\max}(T-\overline{K}Z)^{j}).
  \end{align*}
  For $j=i-1$, 
  \begin{equation*}
    \lVert b_{i,i-1} - s_{i,i-1}\rVert = \lVert ZK_j - Z\overline{K}
    \rVert = O(\lambda_{\max}(T-\overline{K}Z)^j). 
  \end{equation*}
\end{proof}

\begin{proof}[\autoref{thm:vcd-ar}]
  Consider first an AR$(d)$ model and the class
  \[
  \mathcal{C}_{|\F|} = \{ (u_0,u_1,u)\in\R\times\R\times\R^d: \indicator(|u_0-a^\top u|>u_1)=1\}.
  \]
  This is the class contains sets of points $u$ for which linear
  predictions of $u_0$ are made via $a$ with threshold $b$.
  Rewrite $\mathcal{C}_{|\F|}$ as
  \begin{align*}
    \mathcal{C}_{|\F|} &= \left\{(u_0,u_1,u): \max\{[\sgn((1\ -a\ -u_1)^\top(u_0\ u\ 1))],\
      [\sgn((-1\ a\ -u_1)^\top(u_0\ u\ 1))]\}=1\right\}.
  \end{align*}
  Then $\mathcal{C}_{|\F|}$ is a 2-combination of $\F$ as in
  \citet[Theorem 7.3]{Anthony-Bartlett-neural-network-learning}. As $\F$ is a linear function, it has
  solution set components bound $B=1$ \citep[p.\ 91]{Anthony-Bartlett-neural-network-learning}. We add an intercept
  to $\F$ so that it will be closed under addition. Then by
  \citep[Theorem 7.6]{Anthony-Bartlett-neural-network-learning},
  \[
  G(n, \mathcal{C}_{|\F|})  \leq \left(\frac{2e
      n}{d+1}\right)^{d+1}.
  \]
  Then apply \autoref{lem:monotone-vc}.
  The result for VARs is similar. Consider the class
\begin{align*}
  \mathcal{C}_{\mathcal{G}} &= \left\{[\sgn((1\ -a_1\ 0\cdots 0\ -u_1/k)^\top
  (u_0\ u\ 1))] \lor [\sgn((-1\ a_1\ 0\cdots 0\ -u_1/k)^\top
  (u_0\ u\ 1))]\right.\\
  &\lor [\sgn((1\ 0\ -a_2\ 0\cdots 0\ -u_1/k)^\top
  (u_0\ u\ 1))]\lor [\sgn((-1\ 0\ a_2\ 0\cdots 0\ -u_1/k)^\top
  (u_0\ u\ 1))]\\
  & \lor \cdots\lor \left.[\sgn((1\ 0\cdots 0\  -a_k\ -u_1/k)^\top
  (u_0\ u\ 1))]\lor [\sgn((-1\ 0\cdots 0\ a_k\ -u_1/k)^\top
  (u_0\ u\ 1))]\right\}.\\
\end{align*}
This is the class which correctly classifies at least one of the $k$
coordinates. By \citep[Theorem 7.6]{Anthony-Bartlett-neural-network-learning},
\[
G(n, \mathcal{C}_{\mathcal{G}}) \leq \left(\frac{2ke
    n}{k(d+1)}\right)^{k(d+1)} = \left(\frac{2e
    n}{d+1}\right)^{k(d+1)}.
\]
  Then apply \autoref{lem:monotone-vc}. As $|\F| \subseteq
  \mathcal{G}$, we have the result.
\end{proof}

\begin{proof}[\autoref{thm:ssKalman} and \autoref{cor:ssKalman}]
  Let $\F_2$ be the class of predictors given by
  \autoref{alg:ssKalman}. Let $\F_2^{\prime}$ be the same 
  class, but
  predictions are made based on the truncated 
  memory length $d$.  That is for $f_2\in\F_2$,
  \begin{align}
    \widehat{Y}_{d+1:n+1} &= \mathbf{S} \Yin
   \end{align}
  where
  \[
  \mathbf{S} =
  \begin{bmatrix}s_{d+1,1} & \cdots & s_{d+1,d} &  &
    \multicolumn{2}{c}{\multirow{2}{*}{\BigChar{$0$}}}  \\
    s_{d+2,1} & \cdots & s_{d+2,d} & s_{d+2,d+1}  \\
    \vdots &&\vdots&&\ddots\\
    s_{n+1,1} &\cdots &s_{n+1.d} &s_{n+1,d+1}&\cdots & s_{n+1,n}
  \end{bmatrix} 
  =     \begin{bmatrix}s_{d-1} & \cdots & s_{0} &  &
    \multicolumn{2}{c}{\multirow{2}{*}{\BigChar{$0$}}}  \\
    s_{d} & \cdots & s_{1} & s_{0}  \\
    \vdots &&\vdots&&\ddots\\
    s_{n-1} &\cdots &s_{n-d} &s_{n-d+1}&\cdots & s_{0}
  \end{bmatrix} 
    \]
  as in \autoref{eq:smatrix} (redefining $s_{i,j}\rightarrow s_{i-j-1}$), but for $f_2^{\prime} \in \F_2^{\prime}$,
  \begin{align}
    \widehat{Y}'_{d+1:n+1} &= \mathbf{S}' \Yin
  \end{align}
  where
  \begin{align}
    \label{ed:Strunc}
    \mathbf{S}' &= \begin{bmatrix}
      s_{d-1} &s_{d-2}&\cdots & s_{0} & 0 &  \multirow{2}{*}{\BigChar{$0$}}\\
      0& s_{d-1}  & \cdots & s_{1}& s_{0}  \\
      &&\multirow{2}{*}{\BigChar{$0$}}&&\ddots\\
       &&&&&  s_{d-1}&\cdots & s_{0}
    \end{bmatrix}.
  \end{align}
  Note that by using \autoref{alg:ssKalman}, predictions $\hat{Y}'_i$
  are made using the same linear combination of previous observations
  at all times $i$. Thus, truncating the memory, makes these
  predictions simple autoregressive models of order $d$.

  Define $\widetilde{R}_n(f_2^{\prime})$ to be the training error of this truncated
  predictor $f_2^{\prime}$. Then, for any $f_2 \in \F_2$, 
  \begin{align}
    R_n(f_2) - \hat{R}_n(f_2) &= (R_n(f_2)-R_n(f_2^{\prime})) +
                            (R_n(f_2^{\prime}) -
                            \widetilde{R}_n(f_2^{\prime})) + 
                            (\widetilde{R}_n(f_2^{\prime}) -
                            \hat{R}_n(f_2))\\
                          &= \Delta_1(d) +                             (R_n(f_2^{\prime}) -
                            \widetilde{R}_n(f_2^{\prime})) + \Delta_2(d)\\
                          &= (R_n(f_2^{\prime}) -
                            \widetilde{R}_n(f_2^{\prime})) 
                            + \delta_d(f_2),
  \end{align}
  where we have defined $\Delta_1(d)$, $\Delta_2(d)$, and
  $\delta_d(f_2)$ in the obvious way. For now, we will simply proceed
  with the analysis incorporating the approximation term
  $\delta_d(f_2)$, before showing that $\delta_d(f_2)$ decays rapidly for
  this class of models. 

  Now, 
  \begin{equation}
  R_n(f_2) - \hat{R}_n(f_2) - \delta_d(f_2) = R_n(f_2^{\prime}) -
  \widetilde{R}_n(f_2^{\prime}) = R_\dplus(f_2^{\prime}) - \widetilde{R}_n(f_2^{\prime}).
  \end{equation}
  Dividing through by $Q_d(f_2^{\prime})=Q_d(f_2)$ and taking the supremum over
  $\F_2$ (and therefore over $\F_2^{\prime}$) gives 
  \begin{equation}
  \sup_{f_2\in\F_2}\frac{R_n(f_2) -
    \widehat{R}_n(f_2)-\delta_d(f_2)}{Q_d(f_2)} \leq \sup_{f_2 \in
    \F_2}\frac{ R_n(f_2^{\prime}) -
    \widetilde{R}_n(f_2^{\prime})}{Q_d(f_2^{\prime})}
  \end{equation}
  Finally,
  \begin{align}
    \P\left(\sup_{f\in\F}\frac{R_n(f) -
        \widehat{R}_n(f)-\delta_d(f)}{Q_d(f)} > \epsilon\right)
    &\leq \P\left(\sup_{f^{\prime} \in
        \F^{\prime}}\frac{ R_n(f^{\prime}) - \widetilde{R}_n(f^{\prime})} {Q_d(f^{\prime})}
      >\epsilon\right)
  \end{align}
  Since $\F^{\prime}$ is a class with finite memory, we can apply
  \autoref{thm:bound1} and \autoref{cor:bound1a} to get the results.


  To show that the finite approximation $\delta_d(f)$ decays rapidly
  for \autoref{alg:ssKalman}, consider both components separately.
  For the case of the difference in expected risks, we need only consider the
  last rows of $\mathbf{S}$ and $\mathbf{S}'$. As
  \begin{align}
    \Delta_1(d) = R_n(f)-R_n(f^{\prime})  =\E[\mnorm{ Y_{n+1} - \mathbf{s}_{n+1}\Yin}]
    - \E[\mnorm{ Y_{n+1} - \mathbf{s}'_{n+1}\Yin}],
  \end{align}
  where $\mathbf{s}_{n+1}$ indicates the $(n+1)^{st}$ row of $\mathbf{S}$ and
  similarly for $\mathbf{s}'_{n+1}$,
  \begin{align}
    \Delta_1(d)
    &\leq \E[ \mnorm{ (\mathbf{s}_{n+1} - \mathbf{s}'_{n+1})\Yin}]&
    \mbox{(by the Triangle inequality for $\ell$)}\\
    &= \E\left[\sum_{j=1}^{n-d}\ell(s_{n-j}Y_j)\right]\\
  &\leq \sum_{j=1}^{n-d}\ell^*(s_{n-j})\E[\ell(Y_j)] &
      \mbox{$\bigg($as 
      $\ell$ is a norm, with
      $\ell^*(s) 
       = \sup_{x\neq 0}\frac{\ell(s
       x)}{\ell(x)}\bigg)$}\\
    &=\E[\ell(Y_1)]\left( \sum_{j=1}^{n-d}\ell^*(s_{n-j}) \right)&\mbox{(by stationarity)}\\
    &=\E[\ell(Y_1)]\left( \sum_{j=1}^{n-d}O(\rho^{n-j})\right)
                                                                  &\mbox{(\autoref{lem:filterWeightsConv}
                                                                    and
                                                                    equivalence
                                                                    of
                                                                    norms)}\\
  &= O\left(\frac{\rho^{d}-\rho^{n}}{1-\rho}\right)= O(\rho^{d}).
  \end{align}
  Similarly,
  \begin{align}
   \lefteqn{\Delta_2(d) =  \widetilde{R}_n(f') - \hat{R}_n(f)}\\
    &\leq \frac{1}{n-d}\sum_{i=d+1}^{n}\ell((\mathbf{s}_i-\mathbf{s}'_i)
      Y_{1:i-1}) \\
    & = \frac{1}{n-d}\sum_{i=d+2}^{n}\sum_{j=1}^{i-d-1}
      \ell(s_{i-j-1}Y_{j})\\
    & \leq \frac{1}{n-d}\sum_{i=d+2}^{n}O_p(1) \sum_{j=1}^{i-d-1}
      \ell^*(s_{i-j-1})\\ 
    & \leq O_\P((n-d)^{-1})\sum_{i=d+2}^{n} \sum_{j=1}^{i-d-1}
      O(\rho^{i-j-1})\\
    &=O_\P((n-d)^{-1})O\left[\frac{(n-d-1)\rho^{d} -
      (n-d)\rho^{d+1} + \rho^{n}}{(1-\rho)^2}\right]\\
    &= O_\P(\rho^{d}) 
  \end{align}
\end{proof}

\begin{proof}[\autoref{thm:bound3} and \autoref{cor:bound3}]
  Let $\F_1$ be the class of predictors given by
  \autoref{alg:kalman}. Let $\F_2^{\prime}$ be the class of predictors
  given by \autoref{alg:ssKalman} based on the truncated 
  memory length $d$.  That is for $f_1\in\F_1$,
  \begin{align}
    \widehat{Y}_{d+1:n+1} &= \mathbf{B} \Yin
   \end{align}
  where
  \begin{align}
    \mathbf{B} &=
    \begin{bmatrix}b_{d,1} & \cdots & b_{d,d} &  &
      \multicolumn{2}{c}{\multirow{2}{*}{\BigChar{$0$}}}  \\
      b_{d+1,1} & \cdots & b_{d+1,d} & b_{d+1,d+1}  \\
      \vdots &&\vdots&&\ddots\\
      b_{n+1,1} &\cdots &b_{n+1,d} &b_{n+1,d+1}&\cdots & b_{n+1,n}
    \end{bmatrix} 
  \end{align}
  as in \autoref{eq:bmatrix}, but for $f_2^{\prime}$
  \begin{align}
    \widehat{Y}'_{d+1:n+1} &= \mathbf{S}' \Yin
  \end{align}
  as in \autoref{ed:Strunc}.

  Then, proceeding as in the previous proof, define
  $\widetilde{R}_n(f_2^{\prime})$ to be the training error of the
  associated truncated limiting predictor $f'_2$
  and write analogously
 \begin{align}
    R_n(f_1) - \hat{R}_n(f_1) &= (R_n(f_1)-R_n(f_2^{\prime})) +
                            (R_n(f_2^{\prime}) -
                            \widetilde{R}_n(f_2^{\prime})) + 
                            (\widetilde{R}_n(f_2^{\prime}) -
                            \hat{R}_n(f_1))\\
                          &= \Delta_1(d) +                             (R_n(f_2^{\prime}) -
                            \widetilde{R}_n(f_2^{\prime})) + \Delta_2(d)\\
                          &= (R_n(f_2^{\prime}) -
                            \widetilde{R}_n(f_2^{\prime})) 
                            + \delta_d(f_1),\\
   \intertext{where}
   \delta_d(f_1) &= (R_n(f_1)-R_n(f_2^{\prime})) +
                   (\widetilde{R}_n(f_2^{\prime}) - \hat{R}_n(f_1)) =
                   \Delta_1(d) + \Delta_2(d).
  \end{align}

  Now, by \autoref{lem:filterWeightsConv}, the truncated version of $f_1$, $f'_1$
  converges to $f'_2$ so that $f'_1(\Yij{n-d}{n-1})$ converges in
  distribution to $f'_2(\Yij{n-d}{n-1})$. As $\ell(y-x)^2$ is continuous
  in $x$, $\ell(Y_n-f'_1(\Yij{n-d}{n-1}))^2$ converges in distribution
  to $\ell(Y_n-f'_2(\Yij{n-d}{n-1}))^2$ by the continuous mapping
  theorem. By \autoref{ass:C},
  $\ell(Y_n-f'_1(\Yij{n-d}{n-1}))^2$ is uniformly integrable, so
  \[
  \lim_{n\rightarrow\infty}
  \E\left[\ell(Y_n-f'_1(\Yij{n-d}{n-1}))^2\right]=
  \E\left[\ell(Y_n-f'_2(\Yij{n-d}{n-1}))^2\right].
  \]
  
  Therefore,
  \begin{equation}
    \frac{R_n(f_1) -
    \widehat{R}_n(f_1)-\delta_d(f_1)}{\lim_{n\rightarrow\infty}Q_d(f_1)} = \frac{ R_n(f_2^{\prime}) -
    \widetilde{R}_n(f_2^{\prime})}{Q_d(f_2^{\prime})}
  \end{equation}
  For each $f_1\in \F_1$ (parameterized by $\theta$), there exists one
  function $f_2 \in \F_2$ which results from the same parameter
  vector, so that 
  \begin{equation}
   \sup_{f_1\in\F_1} \frac{R_n(f_1) -
    \widehat{R}_n(f_1)-\delta_d(f_1)}{\lim_{n\rightarrow\infty}Q_d(f_1)}
  \leq \sup_{f'_2 \in \F'_2}\frac{ R_n(f_2^{\prime}) -
    \widetilde{R}_n(f_2^{\prime})}{Q_d(f_2^{\prime})}.
  \end{equation}
  Therefore,
  \begin{align}
    \P\left(\sup_{f_1\in\F_1}\frac{R_n(f_1) -
        \widehat{R}_n(f_1)-\delta_d(f_1)}{\lim_{n\rightarrow\infty}Q_d(f_1)} > \epsilon\right)
    &\leq \P\left(\sup_{f_2^{\prime} \in
        \F_2^{\prime}}\frac{ R_n(f_2^{\prime}) - \widetilde{R}_n(f_2^{\prime})} {Q_d(f_2^{\prime})}
      >\epsilon\right)
  \end{align}
  Since $\F_2^{\prime}$ is a class with finite memory, we can apply
  \autoref{thm:bound1} and \autoref{cor:bound1a} to get the results.


  To show that the finite approximation $\delta_d(f_1)$ decays
  rapidly, consider both components separately. 
  For the case of the difference in expected risks, we need only consider the
  last rows of $\mathbf{B}$ and $\mathbf{S}'$. As
  \begin{align}
    \Delta_1(d) = R_n(f_1)-R_n(f_2^{\prime})  =\E[\mnorm{ Y_{n+1} - \mathbf{b}_{n+1}\Yin}]
    - \E[\mnorm{ Y_{n+1} - \mathbf{s}'_{n+1}\Yin}],
  \end{align}
  where $\mathbf{b}_{n+1}$ indicates the $(n+1)^{st}$ row of $\mathbf{B}$ and
  similarly for $\mathbf{s}'_{n+1}$. Then, as in the previous proof,
  \begin{align}
    \Delta_1(d) &\leq \E[ \mnorm{ (\mathbf{b}_{n+1} -
                  \mathbf{s}'_{n+1})\Yin}]\label{eq:ineq1}\\ 
                &\leq \E\left[\left(\sum_{j=1}^{n-d}\ell(b_{n+1,j}Y_j) +
                  \sum_{j=n-d+1}^n
                  \ell((b_{n+1,j}-s_{n-j})Y_j)\right)\right]\label{eq:ineq2}\\
                &\leq \sum_{j=1}^{n-d}\ell^*(b_{n+1,j})\E[\ell(Y_j)] 
                  + \sum_{j=n-d+1}^n \ell^*(b_{n+1,j}-s_{n-j})\E[\ell(Y_j)]\label{eq:ineq3}\\
                &=\E[\ell(Y_1)]\left( \sum_{j=1}^{n-d}\ell^*(b_{n+1,j}) +
                  \sum_{j=n-d+1}^n
                  \ell^*(b_{n+1,j}-s_{n-j})\right)\\
                &=\E[\ell(Y_1)]\left( \sum_{j=1}^{n-d}O(r^{n-j}) +
                  \sum_{j=n-d+1}^nO(\rho^j)\right)\\
                &= O\left(\frac{r^{d}-r^{n}}{1-r}\right) + O\left(\frac{\rho^{n-d+1}(1-\rho^d)}{1-\rho}\right)\\
                &= O(r^{d}) + O(\rho^{n-d+1}).
  \end{align}
  \autoref{eq:ineq1} and \ref{eq:ineq2} are via the triangle inequality for $\ell$, while
  \autoref{eq:ineq3} uses the fact that $\ell$ is a norm, with $\ell^*(b) =
  \sup_{x\neq 0} \frac{\ell(bx)}{\ell(x)}$.
  Similarly,
  \begin{align}
   \lefteqn{\Delta_2(d) =  \widetilde{R}_n(f'_2) - \hat{R}_n(f_1)}\\
    &\leq \frac{1}{n-d}\sum_{i=d+1}^{n}\ell((\mathbf{b}_i-\mathbf{s}'_i)
      Y_{1:n-1}) \\
    & \leq \frac{1}{n-d}\sum_{i=d+1}^{n}\sum_{j=1}^{i-d-1}
      \ell(b_{ij}Y_{j}) + \frac{1}{n-d}\sum_{i=d+1}^{n}
      \sum_{j=i-d}^{i-1} \ell((b_{ij}-s_{n-j})Y_{j}) \\
    & \leq \frac{1}{n-d}\sum_{i=d+2}^{n}O_\P(1) \sum_{j=1}^{i-d-1}
      \ell^*(b_{ij}) + \frac{1}{n-d}\sum_{i=d+1}^{n}O_\P(1)
      \sum_{j=i-d}^{i-1} \ell^*(b_{ij}-s_{ij}) \\
    & \leq O_\P((n-d)^{-1})\sum_{i=d+2}^{n} \sum_{j=1}^{i-d-1}
      r^{i-j-1}+ O_\P((n-d)^{-1})\sum_{i=d+1}^{n}
      \sum_{j=i-d}^{i-1} \rho^j\\
    &=
      O_\P((n-d)^{-1})O\left[\frac{(n-d-1)r^{d}-(n-d)r^{d+1}+r^n}{(1-r)^2}\right]\\
    &\quad + O_\P((n-d)^{-1}) \left[\frac{(1-r^{d})(r^d-r^n)}{r^{d-1}(1-r)^2}\right]\\
    &= O_\P(r^{d}) + O_\P((n-d)^{-1})
  \end{align}
\end{proof}


\begin{proof}[\autoref{thm:mixingoracle}]
  \autoref{thm:boundedbeta} implies that simultaneously
  \begin{align}
    \P\left( R_n(\widehat{f}_{erm}) - \widehat{R}_n(\widehat{f}_{erm}) >
      \epsilon \right) &\leq 8 \GF(n,\CloF)\exp\left\{-
      \frac{\mu\epsilon^2}{\Upsilon} \right\} + 2\mu\beta_{a-d}\\
    \intertext{and}
      \P\left( R_n(f^*) - \widehat{R}_n(f^*)
        > \epsilon \right)
      &\leq 8 \GF(n,\CloF)\exp\left\{-
      \frac{\mu\epsilon^2}{\Upsilon} \right\} + 2\mu\beta_{a-d}.
  \end{align}
  Since $\widehat{R}_n(\widehat{f}_{erm}) -\widehat{R}_n(f^*) \leq 0$, then
  \begin{align}
    \P\left( R_n(\widehat{f}_{erm}) - R_n(f^*)
        > \epsilon \right) &\leq 16 \GF(n,\CloF)\exp\left\{-
      \frac{\mu\epsilon^2}{\Upsilon} \right\} + 4\mu\beta_{a-d}.
  \end{align}

  For any nonnegative random variable $Z$, $\E[Z] = \int_0^\infty
  \P(Z>\epsilon)d\epsilon$, so
  \begin{align}
    L^2_n(\Pi) & = \int_0^\infty d\epsilon     \P\left(
                 (R_n(\widehat{f}_{erm}) - R_n(f^*))^2         > \epsilon
                 \right) \\
               &=\int_0^\xi d\epsilon     \P\left(
                 (R_n(\widehat{f}_{erm}) - R_n(f^*))^2         > \epsilon
                 \right)+\int_\xi^\infty d\epsilon     \P\left(
                 R_n(\widehat{f}_{erm}) - R_n(f^*)         > \epsilon
                 \right)\\
               &=\int_0^\xi d\epsilon     \P\left(
                 (R_n(\widehat{f}_{erm}) - R_n(f^*))^2         > \epsilon
                 \right)+\int_\xi^\infty d\epsilon     \P\left(
                 (R_n(\widehat{f}_{erm}) - R_n(f^*))^2         > \sqrt{\epsilon}
                 \right)\\
               &=\int_0^\xi d\epsilon     \P\left(
                 (R_n(\widehat{f}_{erm}) - R_n(f^*))^2         > \epsilon
                 \right)+\int_\xi^{M^2} d\epsilon     \P\left(
                 R_n(\widehat{f}_{erm}) - R_n(f^*)         > \sqrt{\epsilon}
                 \right)\label{eq:boundedR}\\
               &\leq \xi + \int_\xi^{M^2} d\epsilon\left[ 16 \GF(n,\CloF)\exp\left\{-
                 \frac{\mu\epsilon}{\Upsilon} \right\} +
                 4\mu\beta_{a-d}\right]\\
               &\leq \xi +
                 \int_\xi^\infty d\epsilon 16 \GF(n,\CloF)\exp\left\{-
                 \frac{\mu\epsilon}{\Upsilon} \right\} +
                 \int_\xi^{M^2} d\epsilon 4\mu\beta_{a-d} \\
               &= \xi + \frac{16\Upsilon \GF(n,\CloF)\exp\left\{-
                 \frac{\mu\xi}{\Upsilon} \right\}}{\mu} +
                 4(M^2-\xi)\mu\beta_{a-d}
  \end{align}
  for all $0<\xi<\Upsilon$. Here, \autoref{eq:boundedR} follows
  because $\P\left(
    R_n(\widehat{f}_{erm}) - R_n(f^*)>
    M\right)=0$. Using \autoref{ass:expmix}, take
  \begin{align} a&=(n-d)^{1/(1+\kappa)} = O(n^{1/(1+\kappa)})
    &\mu&=.5(n-d)^{\kappa/(1+\kappa)}=O(n^{\kappa/(1+\kappa)}),\end{align}
  and $\xi=\frac{\log
    \GF(n, \CloF)}{n^{\kappa/(1+\kappa)}}$ to balance the exponential
  and linear terms. Then,
  \begin{equation}
    L_n(\Pi) = O\left(\sqrt{\frac{\log \GF(n,
          \CloF)}{n^{\kappa/(1+\kappa)}}}\right) =
    O\left(\sqrt{\frac{\vcd(\CloF) \log n}{n^{\kappa/(1+\kappa)}}}\right).
  \end{equation}
  as $\CloF$ has finite VC-dimension.

  For the lower bound, apply the i.i.d.\ version, as classification is
  a special case of bounded regression. The result follows.
\end{proof}


\end{document}